\documentclass[psamsfonts]{amsart}

\usepackage{amssymb,amsfonts,amscd}
\usepackage[all,arc]{xy}
\usepackage{enumerate}
\usepackage{mathrsfs}

\newtheorem{thm}{Theorem}[section]
\newtheorem{cor}[thm]{Corollary}
\newtheorem{prop}[thm]{Proposition}
\newtheorem{lem}[thm]{Lemma}
\newtheorem{conj}[thm]{Conjecture}

\theoremstyle{definition}
\newtheorem{defn}[thm]{Definition}

\newtheorem{exmp}[thm]{Example}
\newtheorem{exmps}[thm]{Examples}

\theoremstyle{remark}
\newtheorem{rem}[thm]{Remark}

\makeatletter
\let\c@equation\c@thm
\makeatother
\numberwithin{equation}{section}

\newcommand{\upperRomannumeral}[1]{\uppercase\expandafter{\romannumeral#1}}

\title[]{Ricci curvature and isometric actions with scaling nonvanishing property}

\author[]{Jiayin Pan}
\thanks{The first author enjoyed the hospitality of Capital Normal University at Beijing during the preparation of this paper. Thanks to them.}

\author[]{Xiaochun Rong}
	 \thanks{The second author is partially supported by a research fund from Capital Normal University.}


\newcommand{\Addresses}{{
		\bigskip
		\footnotesize
		
		Jiayin Pan\par\nopagebreak
		\textsc{Department of Mathematics, University of California\ -\ Santa Barbara, CA, USA.}\par\nopagebreak
		\textit{E-mail address}: \texttt{jypan10@gmail.com}
		
		\bigskip
		
		Xiaochun Rong\par\nopagebreak
		\textsc{Department of Mathematics, Rutgers University - New Brunswick, NJ, USA}\par\nopagebreak
		\textsc{Department of Mathematics, Capital Normal University, Beijing, China.}\par\nopagebreak
		\textit{E-mail address}: \texttt{rong@math.rutgers.edu}
}}

\begin{document}

\begin{abstract}
   In the study manifolds of Ricci curvature bounded below, a stumbling obstruction is the
   lack of links between large-scale geometry and small-scale geometry at a fixed reference
   point. There have been few links (volume, dimension) when the unit ball at the point is not collapsed, that is, $\mathrm{vol}(B_1(p))\ge v>0$. In this paper, we conjecture a new link in terms of isometries: if the maximal displacement of an isometry $f$ on $B_1(p)$ is at least $\delta>0$, then the maximal displacement of $f$ on the rescaled unit ball $r^{-1}B_r(p)$ is at least $\Phi(\delta,n,v)>0$ for all $r\in(0,1)$. We call this scaling $\Phi$-nonvanishing property at $p$. We study the equivariant Gromov-Hausdorff convergence of a sequence of Riemannian universal covers with abelian $\pi_1(M_i,p_i)$-actions $(\widetilde{M}_i,\tilde{p}_i,\pi_1(M_i,p_i))\overset{GH}\longrightarrow(\widetilde{X},\tilde{p},G)$, where $\pi_1(M_i,p_i)$-action is scaling $\Phi$-nonvanishing at $\tilde{p_i}$. We establish a dimension monotonicity on the limit group associated to any rescaling sequence. As one of the applications, we prove that for an open manifold $M$ of non-negative Ricci curvature, if the universal cover $\widetilde{M}$ has Euclidean volume growth and $\pi_1(M,p)$-action on $R^{-1}\widetilde{M}$ is scaling $\Phi$-nonvanishing at $\tilde{p}$ for all $R$ large, then $\pi_1(M)$ is finitely generated.
\end{abstract}

\maketitle

We study the fundamental group of a complete $n$-manifold with Ricci curvature bounded below. A major open problem is the Milnor conjecture \cite{Mi68}.

\begin{conj}[Milnor]\label{milnor_conj}
	Let $M$ be an open $n$-manifold of $\mathrm{Ric}\ge 0$, then $\pi_1(M)$ is finitely generated.
\end{conj}

If $M$ has non-negative sectional curvature, then the Milnor conjecture is true; in fact, $M$ is homotopic to a compact manifold \cite{CG72b}. For non-negative Ricci curvature, the Milnor conjecture is difficult due to the absence of a strong relation between large and small scale geometry. Some partial results are proved on this conjecture. Anderson and Li independently proved that any manifold with Euclidean volume growth has a finite fundamental group \cite{An90b,Li86}. Sormani showed that the Milnor conjecture holds if the manifold has small linear diameter growth or linear volume growth \cite{Sor00}. The first author showed that if the Riemannian universal cover of $M$ has Euclidean volume growth and the unique tangent cone at infinity, then the Milnor conjecture is true (see \cite{Pan17b} for a more general statement). In dimension $3$, Liu classified open $3$-manifolds of non-negative Ricci curvature, which confirmed the Milnor conjecture \cite{Liu12}; later, the first author also presented a completely different proof of the Milnor conjecture in dimension $3$ \cite{Pan17}.

Gromov introduced a geometrical method to select a set of generators of $\pi_1(M,p)$ \cite{Gro78}, called \textit{short generators at $p$} and denoted as $S(p)$. Using Toponogov's triangle comparison theorem, Gromov proved that the number of $S(p,R)$ can be uniformly bounded when $M$ has a sectional curvature lower bound, where $S(p,R)$ is the subset of $S(p)$ consisting of elements of length less than $R$ .

\begin{thm}\cite{Gro78}\label{Gro}
	For any $n$ and $R>0$, there exists a constant $C(n,R)$ such that for any complete $n$-manifold $(M,p)$ of $\mathrm{sec}_M\ge -1$, $\#S(p,R)\le C(n,R)$ holds.
\end{thm}

By a scaling trick, Theorem \ref{Gro} implies that the fundamental group of any open $n$-manifold with $\mathrm{sec}\ge 0$ can be generated by at most $C(n,1)$ many elements.

Kapovitch and Wilking proved an estimate of number of short generators for Ricci curvature \cite{KW11}. However, due to the absence of connections between large and small scale geometry, unlike Theorem \ref{Gro}, their result only bounds the number of short generators at some unspecific point $q$ near $p$. If one can bound the number exactly at $p$, then the Milnor conjecture would follow from a scaling trick.

\begin{thm}\cite{KW11}\label{KW}
	For any $n$ and $R>0$, there exists a constant $C(n,R)$ such that for any complete $n$-manifold $(M,p)$ of $\mathrm{Ric}\ge -(n-1)$, there is a point $q\in B_1(p)$ such that $\#S(q,R)\le C(n,R)$.
\end{thm}

For our purpose in this paper, we are confined to abelian fundamental groups. Thanks to Wilking's reduction \cite{Wi00}, to prove Conjecture \ref{milnor_conj} it suffices to consider abelian fundamental group.

Our main result in this paper bounds the number of short generators at $p$ if the unit ball $B_1(\tilde{p})$ in the universal cover $(\widetilde{M},\tilde{p})$ is not collapsed and $\pi_1(M,p)$-action on $(\widetilde{M},\tilde{p})$ satisfies a \textit{scaling non-vanishing property} at $\tilde{p}$, which links large and small scale geometry in term of isometries (see Definition \ref{main_def_nonvanish} below). We conjecture that when $B_1(\tilde{p})$ is not collapsed, the scaling nonvanishing condition is always fulfilled (see Conjecture \ref{main_conj_vol_nv} below); if true, then the scaling nonvanishing assumption in Theorems \ref{main_sg} and \ref{main_milnor} below can be dropped.

\begin{thm}\label{main_sg}
	Given $n,R,v>0$ and a positive function $\Phi$, there exists a constant $C(n,R,v,\Phi)$ such that the following holds.
	
	Let $(M,p)$ be a complete $n$-manifold with abelian fundamental group and
	$$\mathrm{Ric}\ge -(n-1),\quad \mathrm{vol}(B_1({\tilde{p}}))\ge v>0,$$
	where $(\widetilde{M},\tilde{p})$ is the Riemannian universal cover of $(M,p)$. If $\pi_1(M,p)$-action on $\widetilde{M}$ is scaling $\Phi$-nonvanishing at $\tilde{p}$, then $\#S(p,R)\le C(n,R,v,\Phi)$.
\end{thm}

We expect Theorem \ref{main_sg} to be true for general fundamental groups without the abelian assumption (see Remark \ref{rm_abel_nil}). As explained above, Theorem \ref{main_sg} implies a partial result on the Milnor conjecture.

\begin{thm}\label{main_milnor}
	Let $(M,p)$ be an open $n$-manifold with $\mathrm{Ric}\ge 0$. Suppose that the Riemannian universal cover $\widetilde{M}$ has Euclidean volume growth. If there is a positive function $\Phi$ such that the $\pi_1(M,p)$-action on $R^{-1}\widetilde{M}$ is scaling $\Phi$-nonvanishing at $\tilde{p}$ for all $R\ge 1$, then $\pi_1(M)$ is finitely generated.
\end{thm}

We introduce the scaling nonvanishing property. Let $D_{r,p}(f)$ be the displacement of an isometry $f$ of $M$ on $B_r(p)$, that is,
$D_{r,p}(f)=\sup_{q\in B_r(p)}d(f(q),q)$.

\begin{defn}\label{main_def_nonvanish}
	Let $(M,p)$ be a complete Riemannian manifold and $f$ be an isometry of $M$. Let $\Phi(\delta)$ a positive function. We say that $f$ is \textit{scaling $\Phi$-nonvanishing at $p$}, if $s^{-1}D_{s,p}(f)\ge\delta>0$ for some $s\in(0,1]$ implies $r^{-1}D_{r,p}(f)\ge \Phi(\delta)$ for all $r\in(0,s]$. We say that an isometric $G$-action on $M$ is scaling $\Phi$-nonvanishing at $p$, if any element $g\in G$ is scaling $\Phi$-nonvanishing at $p$.
\end{defn}

\begin{conj}\label{main_conj_vol_nv}
	Given $n$ and $v>0$, there is a positive function $\Phi(\delta,n,v)$ such that the following holds.
	
	Let $(M,p)$ be a complete $n$-manifold of
	$$\mathrm{Ric}\ge -(n-1),\quad \mathrm{vol}(B_1(p))\ge v>0.$$
	Then any isometry of $M$ is scaling $\Phi(\delta,n,v)$-nonvanishing at $p$.
\end{conj}

On a fixed complete manifold $(M,p)$ with a non-identity isometry $f$, one can always find a positive function $\Phi$, depending on $(M,p,f)$, such that $f$ is scaling $\Phi$-nonvanishing at $p$. Conjecture \ref{main_conj_vol_nv} seeks a uniform control for a class of manifolds. If $M$ has sectional curvature lower bound $\mathrm{sec}\ge -1$, then any isometry of $M$ is scaling $\Phi(\delta,n)$-nonvanishing at $p$ for all $p\in M$ regardless of the volume condition (see Corollary \ref{nsas_sec}). This relies on the Toponogov's comparison theorem. For Ricci curvature, if one drops the volume lower bound, then Conjecture \ref{main_conj_vol_nv} would fail. For instance, \cite{CC97} constructed a sequence of complete $n$-manifolds $(M_i,p_i)$ with Ricci lower bounds Gromov-Hausdorff converging to a horn $(Y,p)$,
where $Y$ has dimension $5$ but the tangent cone at $p$ is a half line. For such a sequence, one can find a sequence of isometries $f_i$ of $M_i$ fixing $p_i$ such that $D_{1,p_i}(f_i)=\delta>0$ but $r_i^{-1}D_{r_i,p_i}(f_i)\to 0$ for some $r_i\to 0$. We mention that due to relative volume comparison, to verify Conjecture \ref{main_conj_vol_nv}, one only need to check the case $s=1$ in Definition \ref{main_def_nonvanish} (also see Remark \ref{rem_rescale_id}).

Another supporting evidence of Conjecture \ref{main_conj_vol_nv} is that the displacement function of any isometric group action is scaling $\Phi(\delta,n,v)$-nonvanishing. Indeed, we can choose $\Phi$ as a constant function and this result also applies to non-collapsed Ricci limit spacecs.

\begin{thm}\label{main_no_small_subgroup}
	Given $n,v>0$, there exists a positive constant $\delta(n,v)$ such that for any non-collapsed Ricci limit space $(X,p)\in\mathcal{M}(n,-1,v)$
	and any nontrivial subgroup $H$ in $\mathrm{Isom}(X)$, $r^{-1}D_{r,p}(H)\ge \delta$ holds for all $r\in(0,1]$, where $D_{r,p}(H)=\sup_{h\in H} D_{r,p}(h)$.
\end{thm}

We refer Theorem \ref{main_no_small_subgroup} as a quantitative version of \textit{no small subgroup property}, in the sense that there is no nontrivial isometric group action with very small displacement on the unit ball.
Theorem \ref{main_no_small_subgroup} can be further extended from subgroups to \textit{almost subgroups}, whose orbits at every point $q\in B_1(p)$ behaves similarly to subgroups (see Theorem \ref{main_nsas_local}). For more discussions on Conjecture \ref{main_conj_vol_nv}, see Section 2.2. We also prove some applications of Theorem \ref{main_no_small_subgroup} to the structure of fundamental groups in Sections 2.3 and 2.4.

We introduce the main technical result in this paper. Consider an equivariant Gromov-Hausdorff convergence of complete $n$-manifolds $(M_i,p_i)$ of $\mathrm{Ric}\ge -(n-1)$ and its rescaling sequence:
\begin{center}
	$\begin{CD}
	(\widetilde{M}_i,\tilde{p}_i,\Gamma_i) @>GH>> (\widetilde{X},\tilde{p},G)\\
	@VV\pi_iV @VV\pi V\\
	(M_i,p_i) @>GH>> (X,p),
	\end{CD}\qquad \begin{CD}
	(r_i\widetilde{M}_i,\tilde{p}_i,\Gamma_i) @>GH>> (\widetilde{X}',\tilde{p}',G')\\
	@VV\pi_iV @VV\pi' V\\
	(r_iM_i,p_i) @>GH>> (X',p'),
	\end{CD}\quad (*)$
\end{center}
where $\Gamma_i=\pi_1(M_i,p_i)$ and $r_i\to\infty$. The limit group $G$ (resp. $G'$), as a closed subgroup of $\mathrm{Isom}(\widetilde{X})$ (resp. $\mathrm{Isom}(\widetilde{X}')$), is a Lie group \cite{CC00a,CN12}. It follows from Theorem \ref{main_no_small_subgroup} that with a lower bound on $B_1(\tilde{p}_i)$, if $G=\{e\}$, then $G'=\{e\}$. We prove the following connections between $G$ and $G'$.

\begin{thm}[Dimension monotonicity of symmetries]\label{main_dim}
	Let $(M_i,p_i)$ be a sequence of complete $n$-manifolds with abelian fundamental groups $\Gamma_i$ and
	$$\mathrm{Ric}_{M_i}\ge -(n-1),\quad \mathrm{vol}(B_1(\tilde{p}_i))\ge v>0.$$
	Consider the convergent sequence and any rescaling sequence as in ($*$).
	If there is a positive function $\Phi$ such that $\Gamma_i$-action on $\widetilde{M}$ is scaling $\Phi$-nonvanishing at $\tilde{p}$ for all $i$, then\\
	(1) $\dim(G')\le\dim(G),$\\
	(2) If $G'$ has a compact subgroup $K'$, then $G$ contains a subgroup $K$ fixing ${p}$ and $K$ is isomorphic to $K'$.
\end{thm}

Note that because $\dim(\widetilde{X})=\dim(\widetilde{X}')=n$, (1) is equivalent to a dimension monotonicity of spaces $\dim(X)\le\dim(X')$. We mention that volume assumption in Theorem \ref{main_dim} can be replaced by a no small almost subgroup condition on $B_1(\tilde{p})$, with which $\widetilde{X}$ and $\widetilde{X}'$ may be collapsed in general (see Theorem \ref{dimension'}); this also leads to a bound on the number of short generators and finite generation with a no small subgroup condition (see Theorems \ref{sg_nsas} and Theorem \ref{milnor_nsas}).

We indicate our approach to Theorem \ref{main_dim}. A crucial consequence of volume and the scaling nonvanishing property plays a key rule in proving Theorem \ref{main_dim}: if a subset $A$ of $\Gamma$ has orbit $A\tilde{p}$ similar to a group action orbit, then its displacement cannot be too small (compare with Theorem \ref{main_no_small_subgroup}). More precisely, if a sequence of $A_i\subseteq \Gamma_i$ with $A_i^{-1}=A_i$ satisfies
$$\dfrac{d_H(A_i\tilde{p}_i,A_i^2\tilde{p}_i)}{\mathrm{diam}(A_i\tilde{p}_i)}\to 0,$$
then $D_{1,\tilde{p}_i}(A_i)\not\to 0$, where $A^k=\{a^k|a\in A\}$ and $d_H$ is the Hausdorff distance on $\widetilde{M_i}$. Note that if the above ratio is small, then $A\tilde{p}$ is similar to $A^2\tilde{p}$ and we see the orbit is close to a group action orbit. We call this \textit{no small almost subgroup property} at $\tilde{p}$ (see Section 2.2 for more details).

For Theorem \ref{main_dim}, let us first consider an easy case:
$G=\mathbb{R}$ and $G'=\mathbb{R}\times S^1$. For simplicity, we also assume that $S^1$-action is free at some $\tilde{p}'$. Let $\gamma$ be the element of order $2$ in $S^1$ and $\gamma_i\in \Gamma_i$ such that
$$(r_i\widetilde{M}_i,\tilde{p}_i,\gamma_i)\overset{GH}\longrightarrow (\widetilde{X}',\tilde{p}',\gamma).$$
Put $A_i=\{e,\gamma_i^{\pm 1}\}$. Then with respect to the above sequence, $A_i\overset{GH}\to \langle\gamma\rangle$ and thus the scaling invariant
$$\dfrac{d_H(A_i\tilde{p}_i,A_i^2\tilde{p}_i)}{\mathrm{diam}(A_i\tilde{p}_i)}\to 0.$$
Note that before rescaling $D_{1,\tilde{p}_i}(A_i)\to 0$, a contradiction to no small almost subgroup property at $\tilde{p}_i$. Next we consider a typical situation: $G=\mathbb{R}$ and $G'=\mathbb{R}^2$.
The difficulty compared with the previous case is that,
there is no indication on how to choose a sequence of collapsed almost subgroups from $G'=\mathbb{R}^2$.
Our strategy is finding a suitable intermediate rescaling sequence,
from which we are able to pick up a sequence of small almost groups (see Section 3 for details). This method of choosing an intermediate rescaling sequence is also used in \cite{Pan17b}.

We also roughly illustrate the proof of Theorem \ref{main_sg} by assuming Theorem \ref{main_dim}. Suppose that there is a contradicting sequence:
\begin{center}
	$\begin{CD}
	(\widetilde{M}_i,\tilde{p}_i,\Gamma_i) @>GH>> (\widetilde{X},\tilde{p},G)\\
	@VV\pi_iV @VV\pi V\\
	(M_i,p_i) @>GH>> (X,p)
	\end{CD}$
\end{center}
satisfying the following conditions:\\
(1) $\mathrm{Ric}_{M_i}\ge -(n-1)$, $\mathrm{vol}(B_1(\tilde{p}_i))\ge v>0$;\\
(2) $\pi_1(M_i,p_i)$ is abelian, whose action is scaling $\Phi$-nonvanishing at $\tilde{p}_i$;\\
(3) $\#S(p_i,R)\to\infty$.\\
Roughly speaking, we derive a contradiction by induction on the dimension of $G$. Assume that $\dim(G)=0$, that is, $G$ is discrete.
Recall that there is a sequence of $\epsilon_i$-equivariant maps \cite{FY92}
$$\psi_i: \Gamma_i(R)\to G(R),\quad \Gamma_i(R)=\{\gamma\in\Gamma_i\ |\ d(\gamma\tilde{p}_i,\tilde{p}_i)\le R\}$$
for some $\epsilon_i\to 0$. By the discreteness of $G$ and Theorem \ref{main_no_small_subgroup},
it is not difficult to check that $\#\Gamma_i(R)$ is uniformly bounded (see Corollary \ref{stable_nss} for details), thus $\#S(p_i,R)$ is uniformly bounded, a contradiction to (3). Assume that there is no such contradicting sequence with $\dim(G)\le k$,
while there is one with $\dim(G)=k+1$. We shall obtain a contradiction by constructing a new contradicting sequence with $\dim(G)\le k$. For a sequence $m_i\to\infty$, let $\Gamma_{i,m_i}$ be the subgroup of $\Gamma_i$ generated by the first $m_i$ short generators at $p_i$. If for some $m_i\to\infty$,
$$(\widetilde{M}_i,\tilde{p}_i,\Gamma_{i,m_i})
\overset{GH}\longrightarrow(\widetilde{X},\tilde{p},H).$$
and $\dim(H)\le k$, then we are done. Without lose of generality, we assume that $\dim(H)=k+1$ for all $m_i\to\infty$. For some $m_i\to\infty$ with $|\beta_i|\to 0$, where $\beta_i=\gamma_{i,m_i+1}$ is the $(m_i+1)$-th short generator in $\Gamma_i$, we consider a sequence of intermediate coverings,
\begin{center}
	$\begin{CD}
	(\widetilde{M}_i,\tilde{p}_i,\langle\Gamma_{i,m_i},\beta_i\rangle)
	@>GH>> (\widetilde{X},\tilde{p},K)\\
	@VV\pi_iV @VV\pi V\\
	(\overline{M}_i=\widetilde{M}_i/\Gamma_{i,m_i},\bar{p}_i,\langle\overline{\beta_i}\rangle)
	@>GH>> (\overline{X},\bar{p},\Lambda).
	\end{CD}$
\end{center}
Because $d(\beta_i\tilde{p}_i,\tilde{p}_i)\to 0$
and $\dim(H)=\dim(K)$, one can show that $\Lambda$ is discrete and fixes $\bar{p}$. Put $r_i=\mathrm{diam}(\langle\overline{\beta_i}\rangle\bar{p}_i)\to 0$ and consider the rescaling sequences
\begin{center}
	$\begin{CD}
	(r_i^{-1}\widetilde{M}_i,\tilde{p}_i,\Gamma_{i,m_i},\langle\Gamma_{i,m_i},\beta_i\rangle)
	@>GH>> (\widetilde{X}',\tilde{p}',H',K')\\
	@VV\pi_iV @VV\pi V\\
	(r_i^{-1}\overline{M}_i,\bar{p}_i,\langle\overline{\beta_i}\rangle)
	@>GH>> (\overline{X}',\bar{p}',\Lambda').
	\end{CD}$
\end{center}
By Theorem \ref{main_dim}, $\dim(K')\le\dim(K)=k+1$. If $\dim(H')<\dim(K')$, then we reduce the dimension successfully. One can check that $(r_i^{-1}\overline{M}_i,\bar{p}_i)$ is a desired contradicting sequence. If $\dim(H')=\dim(K')$, then we apply Theorem \ref{main_dim}(2) and use an induction argument on the number of the connected components of the isotropy subgroup at $\tilde{p}'$ (see Section 4 for details).

\tableofcontents

\section{Preliminaries}

For convenience of readers, we provide some basic notions and properties that will be used in this paper.

Given $n$ and $v>0$,
let $\mathcal{M}(n,-1)$ be the set of all limit spaces of sequences of complete $n$-manifolds $(M_i,p_i)$ of
$$\mathrm{Ric}_{M_i}\ge -(n-1);$$
let $\mathcal{M}(n,-1,v)$ be the set of all limit spaces of sequences of $n$-manifolds $(M_i,p_i)$ with curvature condition above and
$$\mathrm{vol}(B_1(p_i))\ge v>0.$$

Let $(X,x)\in\mathcal{M}(n,-1)$ and given any sequence $r_i\to\infty$, passing to a subsequence if necessary,
$$(r_iX,x)\overset{GH}\longrightarrow(C_x X,o).$$
We call $(C_x X,o)$ a tangent cone at $x$.
In general, tangent cones at $x$ may not be unique; they may not have the same Hausdorff dimension \cite{CC97}. For non-collapsed Ricci limit spaces, the tangent cones must be metric cones \cite{CC96}.

\begin{thm}\cite{CC96}\label{CC_cone}
If $(X,x)\in\mathcal{M}(n,-1,v)$, then any tangent cone $(C_x X,o)$ is an $n$-dimensional metric cone $C(Z)$ with vertex $o$ and $\mathrm{diam}(Z)\le \pi$.
\end{thm}



Recall that by Bishop volume comparison, any metric ball $B_1(p)$ in a complete $n$-manifold $M$ of $\mathrm{Ric}\ge 0$ has volume at most $\mathrm{vol}(B_1^n(0))$, the volume of the unit ball in the $n$-dimensional Euclidean space. Moreover, $B_1(p)$ attains maximal volume if and only if $B_1(p)$ is isometric to $B_1^n(0)$. Cheeger and Colding proved a quantitative version of this volume rigidity result.

\begin{thm}\cite{CC96}\label{CC_regular_vol}
	There exists a positive function $\Phi(\delta|n)$ with $\lim\limits_{\delta\to 0^+}\Phi(\delta|n)=0$ such that the following holds.
	
	Let $(M,p)$ be a complete $n$-manifold with $\mathrm{Ric}\ge -(n-1)\delta$.\\
	(1) If $$d_{GH}(B_1(p),B_1^n(0))\le\delta,$$ then
	$$\mathrm{vol}(B_1(p))\ge (1-\Psi(\delta|n))\mathrm{vol}(B_1^n(0)).$$\\
	(2) If $$\mathrm{vol}(B_1(p))\ge (1-\delta)\mathrm{vol}(B_1^n(0)),$$ then $$d_{GH}(B_1(p),B_1^n(0))\le\Psi(\delta|n).$$
\end{thm}

We also recall Gromov's short generators \cite{Gro78}:
\begin{defn}\label{Gromov_sg}
Let $(\widetilde{M},\tilde{p})$ be the Riemannian universal cover of $(M,p)$. A subset $S(p)=\{\gamma_1,\gamma_2...,\}$ of $\pi_1(M,p)$ is called a set of \textit{short generators}, if
\begin{center}
	$d(\gamma_1\tilde{p},\tilde{p})\le d(\gamma\tilde{p},\tilde{p})$ for all $\gamma\in\pi_1(M,p)$,
\end{center}
and for each $k\ge 2$,
\begin{center}
	$d(\gamma_k\tilde{p},\tilde{p})\le d(\gamma\tilde{p},\tilde{p})$ for all $\gamma\in\pi_1(M,p)-\langle\gamma_1,...,\gamma_{k-1}\rangle$,
\end{center}
where $\langle\gamma_1,...,\gamma_{k-1}\rangle$ is the subgroup generated by $\gamma_1,...,\gamma_{k-1}$.
\end{defn}





\section{Curvature, volume, and isometric group actions}

In this section, we explore equivariant Gromov-Hausdorff convergence with lower bounds on Ricci curvature and volume. We first prove no small subgroup property (Theorem \ref{main_no_small_subgroup}) in Section 2.1. In Section 2.2, we prove an extension of the no small subgroup property (Proposition \ref{main_nsas_local}). Then we introduce the scaling non-vanishing condition and show the connections among these conditions. In Sections 2.3 and 2.4, we apply the no small subgroup property to obtain some structure results on fundamental groups.

\subsection{No small subgroup}

A classical result in Lie group theory says that a topological group is a Lie group if and only if it has no small subgroups: if a subgroup $H$ of a group $G$ is contained in a sufficiently small neighborhood of the identity element, then $H$ is trivial. \cite{CC00a,CN12} showed that for any $X\in \mathcal{M}(n,-1)$, its isometry group $\mathrm{Isom}(X)$ is a Lie group, by ruling out non-trivial small subgroups of $\mathrm{Isom}(X)$. More precisely, they showed that for $X\in\mathcal{M}(n,-1)$, if there is a sequence of subgroups $H_i$ of $\mathrm{Isom}(X)$ such that
$D_{R,x}(H_i)\to 0$
for all $R>0$ and $x\in X$, then $H_i=\{e\}$ for $i$ large, where
$$D_{R,x}(H)=\sup_{f\in H,q\in B_r(x)}d(fq,q).$$
In this subsection we prove Theorem \ref{main_no_small_subgroup}, a quantitative version of no small subgroup property for non-collapsing Ricci limit spaces. We start with a characterization of the identity map on Ricci limit spaces.

\begin{lem}\label{char_id}
	Let $(X,p)\in\mathcal{M}(n,-1)$ be a Ricci limit space. If $g\in\mathrm{Isom}(X)$ has trivial action on $B_s(p)$ for some $s>0$, then $g=e$.
\end{lem}

\begin{proof}
	Scaling the metric if necessary, we can assume that $s=1$. 
	
	The proof is a modification of the arguments of Theorem 4.5 in \cite{CC00a} and Theorem 1.14 in \cite{CN12}. Let $k$ be the dimension of $X$ in the Colding-Naber sense \cite{CN12} and $\mathcal{R}^k$ be the set of points of which any tangent cone is isometric to $\mathbb{R}^k$. Let $\mathcal{R}^k_{\epsilon,\delta}$ be the effective regular set defined as the set of all points $y\in X$ such that
	$$d_{GH}(B_r(y),B^k_r(0))\le \epsilon r$$
	for all $0<r<\delta$ \cite{CC97}.
	
    We recall the uniform Reifenberg property proved in \cite{CN12}: almost every $y\in\mathcal{R}^k$ and almost every $z\in\mathcal{R}^k$ have the property that for any $\epsilon>0$, there exist $\delta>0$ and a geodesic $\gamma_{yz}$ connecting $y$ and $z$ such that $\gamma_{yz}\subseteq \mathcal{R}^k_{\epsilon,\delta}$.
	
	Suppose that $g$ is not the identity element. Let $H$ be the closure of the subgroup generated by $g$, then clearly $H|_{B_1(x)}=\mathrm{id}$. Since $H\not=\{e\}$, for any $\epsilon>0$, there exist $\theta\in(0,\epsilon)$ and a $k$-regular point $w\in(\mathcal{R}_k)_{\epsilon,\theta}$ such that
	$$\theta^{-1}D_{\theta,w}(H)\ge 1/20.$$
	On the other hand, because $H$ acts trivially on $B_1(x)$, there are $\eta>0$ and a $k$-regular point $y\in B_{1/2}(x)\cap (\mathcal{R}_k)_{\epsilon,\eta}$ with
	$$\eta^{-1}D_{\eta,y}(H)=0.$$
	We further assume that the points $w$ and $y$ chosen above satisfy the uniform Reifenberg property, that is, there are $\lambda<\min\{\theta,\eta\}$ and such that $\gamma_{wy}$ lies in $\mathcal{R}^k_{\epsilon,\lambda}$.
	
	If $\lambda^{-1} D_{\lambda,w}(H)\le 1/20$, then by intermediate value theorem, we can find $r\in [\lambda,\eta]$ such that
	$$r^{-1} D_{r,w}(H)=1/20.$$
	
	If $\lambda^{-1} D_{\lambda,w}(H)> 1/20$, together with 
	$$\lambda^{-1}D_{\lambda,y}(H)=0<1/20,$$
	we can find $z$ along $\gamma_{wy}$ such that
	$$\lambda^{-1} D_{\lambda,z}(H)=1/20.$$
	
	Replace the arbitrary $\epsilon>0$ by a sequence $\epsilon_{i}\to 0$. Then we can find $\tau_{i}\ge r_{i}\to 0$, $z_{i}\in\mathcal{R}^k_{\epsilon_{i},\tau_{i}}$ such that $$D_{r_{i},z_{i}}(H)=r_i/20.$$ Consequently,
	$$(r_{i}^{-1}B_{r_{i}}(z_{i}),z_{i},H)\overset{GH}\longrightarrow(B_1^k(0),0,H_\infty)$$
	with $D_{1,0}(H_\infty)=1/20$. However, there is no such a subgroup $H_\infty$ of $\mathrm{Isom}(\mathbb{R}^k)$, a contradiction.
\end{proof}



Next we prove Theorem \ref{main_no_small_subgroup}.

\begin{proof}[Proof of Theorem \ref{main_no_small_subgroup}]
	We show that $D_{1,p}(H)\ge\delta(n,v)$. For the general result $D_{r,p}(H)\ge r\delta(n,v)$ for all $r\in(0,1]$, with a possibly different $\delta$, we can scale the metric by $r^{-1}$. By relative volume comparison on $(r^{-1}X,p)$ the unit ball has volume
	$\mathrm{vol}(r^{-1}B_r(p))\ge C(n)v$
	and thus $D_{1,p}(H)\ge\delta(n,C(n)v)$ on $(r^{-1}X,p)$ for all $r\in(0,1]$. Scaling the metric back to $(X,p)$, we have $D_{r,p}(H)\ge\delta(n,C(n)v)r$.
	
	Now suppose that the contrary holds, then there exists a sequence of spaces $(X_i,p_i)\in\mathcal{M}(n,-1,v)$ and nontrivial subgroups $H_i$ of $\mathrm{Isom}(X_i)$ with $$D_{1,p_i}(H_i)\to 0.$$ By Lemma \ref{char_id}, passing to a subsequence if necessary,
	$$(X_i,p_i,H_i)\overset{GH}\longrightarrow(X,p,\{e\}).$$
	
	We will find a subsequence $i(j)$, $\epsilon_{j(i)}\to 0$, $\tau_{i(j)}\ge r_{i(j)}>0$, $z_{i(j)}\in\mathcal{R}_{\epsilon_{i(j)},\tau_{i(j)}}\subseteq X_{i(j)}$ and $D_{r_{i(j)},z_{i(j)}}(H_{i(j)})=\dfrac{1}{20}r_{i(j)}$. Then
	$$(r_{i(j)}^{-1}B_{r_{i(j)}}(z_{i(j)}),z_{i(j)},H_{i(j)})\overset{GH}\longrightarrow(B_1^n(0),0,H_\infty)$$
	with $D_1(H_\infty)=1/20$; the desired contradiction follows.
	
	Fix a regular point $y\in B_1(p)\subset X$. For each $\epsilon>0$, there is $\delta>0$ such that $y\in\mathcal{R}_{\epsilon,\delta}$. Pick a sequence of regular points $y_i\in X_i$ converging to $y$. Put
	$$\eta_i=d_{GH}(\delta^{-1}B_\delta(y_i),\delta^{-1}B_\delta(y))\to 0.$$
	Because $y\in\mathcal{R}_{\epsilon,\delta}$,
	$$d_{GH}(\delta^{-1}B_\delta(y_i),B_1^n(0))\le\eta_i+\epsilon.$$
	By Theorem \ref{CC_regular_vol}, for all $0<s\le\delta$,
	$$d_{GH}(s^{-1}B_s(y_i),B_1^n(0))\le\Phi(\eta_i+\epsilon,\delta|n).$$
	In other words, $y_i\in\mathcal{R}_{\Phi_i,\delta}$. Also, because $H_i\to\{e\}$,
	$$\delta^{-1}D_{\delta,y_i}(H_i)\to 0.$$
	
	For each $\epsilon$, pick $i(\epsilon)$ large such that for all $i\ge i(\epsilon)$, we have
	$$\eta_i\le\delta \text{\ and\ } \delta^{-1}D_{\delta,y_i}(H_i)\le \dfrac{1}{20}.$$
	
	Now consider a sequence $\epsilon_j\to 0$, then $y\in\mathcal{R}_{\epsilon_j,\delta_j}$ for some $\delta_j\to 0$. There is a subsequence $i(j)$ such that\\
	1. $\eta_{i(j)}\le\delta_j$, thus $y_{i(j)}\in\mathcal{R}_{\Phi_{i(j)},\delta_j}$, where $\Phi_{i(j)}=\Phi(\delta_j+\epsilon_j|n);$\\
	2. $D_{\delta_j,y_{i(j)}}(H_{i(j)})\le \dfrac{1}{20}\delta_j$.
	
	On each $X_{i(j)}$, there is $\theta_{i(j)}>0$, $w_{i(j)}\in\mathcal{R}_{\Phi_{i(j)},\theta_{i(j)}}$ such that
	$$D_{\theta_{i(j)},w_{i(j)}}(H_{i(j)})\ge \dfrac{1}{20}\theta_{i(j)}.$$
	
	The remaining proof is essentially the same as Theorem 4.5 in [CC00a].
\end{proof}

Using Theorem \ref{main_no_small_subgroup}, we prove a corollary below on a convergent sequence with discrete limit group. For $r>0$ and an isometric $G$-action on a space $(X,p)$, we put $G(r)$ as all the elements of $G$ with displacement at $p$ being less than $r$:
$$G(r)=\{g\in G\ |\ d(gp,p)\le r\}.$$

\begin{cor}\label{stable_nss}
Let $(M_i,p_i)$ be a sequence of complete $n$-manifolds with $$\mathrm{Ric}_{M_i}\ge-(n-1),\quad\mathrm{vol}(B_1({{p}}_i))\ge v>0.$$
Suppose that there is an isometric $H_i$-action on $M_i$ for each $i$ and the following sequence converges:
$$(M_i,p_i,H_i)\overset{GH}\longrightarrow (X,p,H).$$
If $H$ is discrete, then 
$$\#H_i(1)\le \#H(2)<\infty$$
for all $i$ large.
\end{cor}

\begin{proof}
	We first show that if a sequence $h_i\in H_i$ with $h_i\overset{GH}\to e$, then $h_i=e$ for all $i$ large. Indeed, because $H$ is a discrete group, it is clear that the group generated by $h_i$ also converges to $\{e\}$. On the other hand, every nontrivial subgroup of $H_i$ has displacement at least $\delta(n,v)$ on $B_1({p}_i)$. Therefore, the subgroup generated by $h_i$ must be trivial and thus $h_i=e$. 
	
	This implies that if two sequences $h_i\overset{GH}\to g$ and $h'_i\overset{GH}\to g$ with $g\in H(2)$, then $h_i=h'_i$ for all $i$ large. Hence
	$$\#H_i(1)\le \# H(2) <\infty$$
	for all $i$ large.
\end{proof}

\subsection{No small almost subgroup and scaling nonvanishing isometries}

We explore the relations among volume, no small almost subgroup property and scaling nonvanishing property in this section. We present two statements equivalent to Conjecture \ref{main_conj_vol_nv} in terms of Gromov-Hausdorff convergence (see Proposition \ref{rescale_identity} and Remark \ref{rem_rescale_id}). We also show that scaling nonvanishing property holds when sectional curvature has a lower bound (Corollary \ref{nsas_sec}).

We first extend the idea of \textit{no small subgroups} to certain subsets that are very close to being subgroups, which we call \textit{almost subgroups}.

\begin{defn}
	Let $G$ be a group and $A$ be a subset of $G$. We say that $A$ is a \textit{symmetric subset} of $G$ if $e\in A$ and $A^{-1}=A$, where $A^{-1}=\{a^{-1}|a\in A\}$.
\end{defn}

\begin{defn}\label{def_almost_subgp}
	Let $\eta>0$. Let $(M,p)$ be a complete $n$-manifold and $G$ be group acting isometrically on $M$. We say that a symmetric subset $A\not=\{e\}$ of $G$ is a $\eta$-subgroup at $p$, if $\mathrm{diam}(Ap)\in (0,\infty)$ and
	$$\dfrac{d_{H}(Ap,A^2p)}{\mathrm{diam}(Ap)}<\eta.$$
	We say that $A$ is a $\eta$-subgroup on $B_1(p)$, if $\mathrm{diam}(Aq)\in (0,\infty)$ and
	$$\dfrac{d_{H}(Aq,A^2q)}{\mathrm{diam}(Aq)}<\eta$$
	for all $q\in B_1(p)$.
\end{defn}

Note that in Definition \ref{def_almost_subgp}, if the ratio is $0$, then $Ap=A^2p$ and thus $A$-orbit at $p$ is a group action orbit. Therefore, this ratio describes how close a symmetric subset $A$ is to being a subgroup regarding its orbit at $p$. 

We introduce the notion of no small almost subgroup at a point, or on a metric ball.

\begin{defn}\label{main_def_nsas}
	Let $\epsilon,\eta,r>0$ and $(M,p)$ be an $n$-manifold. For a subgroup $G$ of $\mathrm{Isom}(M)$ acting on $M$,
	we say that $G$-action has \textit{no $\epsilon$-small $\eta$-subgroup} at $p$ (resp. on $B_1(p)$) with scale $r$, if any $\eta$-subgroup $A$ at $p$ (resp. on $B_1(p)$) satisfies
	$$r^{-1}B_{r,p}(A)\ge \epsilon.$$
\end{defn}

We show that Theorem \ref{main_no_small_subgroup} implies no $\epsilon$-small $\eta$-subgroup on $B_1(p)$, where $\epsilon$ and $\eta$ only depend on $n$ and $v$.

\begin{prop}\label{main_nsas_local}
	Given $n, v>0$, there exist positive constants $\epsilon(n,v)$ and $\eta(n,v)$ such that the following holds.
	
	Let $(M,p)$ be a complete $n$-manifold with
	$$\mathrm{Ric}\ge -(n-1), \quad \mathrm{vol}(B_1({p}))\ge v.$$ For any isometric $G$-action on $M$, $G$-action has no $\epsilon$-small $\eta$-subgroup on $B_1(p)$ with scale $r\in(0,1]$.
\end{prop}

\begin{proof}
	If we can prove a lower bound for $D_{1,p}(A)$, where $A$ is any symmetric subset of $G$, then the estimate of $r^{-1}D_{r,p}(A)$ follows from relative volume comparison. We bound $D_{1,p}(A)$ by a contradicting argument. 
	
	Suppose that there is a sequence of complete $n$-manifolds $(M_i,p_i)$ with $$\mathrm{Ric}_{M_i}\ge-(n-1),\quad\mathrm{vol}(B_1({{p}}_i))\ge v,$$
	and a sequence of symmetric subsets $A_i\not=\{e\}$ of $G_i$ with $D_{1,p_i}(A_i)\to 0$ and
	$$\sup_{q\in B_1({p}_i)}\dfrac{d_{H}(A_iq,A_i^2q)}{\mathrm{diam}(A_iq)}\to 0.$$
	
	For simplicity, we write $D_{1,p_i}$ as $D_1$ since the base point is clear. Let $\delta=\delta(n,v)$ be the constant in Theorem \ref{main_no_small_subgroup}. For any positive integer $j$, we can choose $i(j)$ large with
	$$D_{1}(A_{i(j)})\le \delta/2.$$
	For each $i(j)$, because $\mathrm{diam}(A_{i(j)}p_{i(j)})>0$, we have
	$$r^{-1}D_{r}(A_{i(j)})\to\infty$$
    as $r\to 0$. By intermediate value theorem, there is $r(j)\in(0,1]$ such that $$r(j)^{-1}D_{r(j)}(A_{i(j)})=\delta/2.$$
	For simplicity, we just call $r(j)$ as $r_i$ and the subsequence $i(j)$ as $i$. It is clear that $r_i\to 0$ by Lemma \ref{char_id}. 
	
	After rescaling $r_i^{-1}\to\infty$,
	$$(r_i^{-1}{M}_i,{p}_i,A_i)\overset{GH}\longrightarrow({X}',{p}',A_\infty).$$
	$A_\infty$ satisfies $D_1(A_\infty)=\delta/2$. By Theorem \ref{main_no_small_subgroup}, $A_\infty$ is not a subgroup. Thus there is some point $q\in B_1({p}')$ such that $A_\infty^2q\not=A_\infty q$ (see Lemma \ref{char_id}).
	
	On the other hand, we know that
	$$\sup_{q\in B_{r_i}({p}_i)}\dfrac{d_{H}(A_iq,A_i^2q)}{\mathrm{diam}(A_iq)}\to 0.$$
	For a sequence $q_i$ converging to $q$,
	$$r_i^{-1}d_{H}(A_iq_i,A_i^2q_i)\le\epsilon_i\cdot r_i^{-1}D_{1}(A_i)=\epsilon_i\cdot\delta/2\to 0$$
	for some sequence $\epsilon_i\to 0$. Thus $A_\infty q=A_\infty^2q$, a contradiction.
\end{proof}

Next we show that whether a sequence $(M_i,p_i,G_i)$ has no small almost subgroup at $p_i$ is closely related to almost identity maps on different scales. Note that (2) below naturally leads to the scaling nonvanishing property (see Remark \ref{rem_rescale_id}).

\begin{prop}\label{rescale_identity}
	Let $(M_i,p_i)$ be a sequence of complete $n$-manifolds with $$\mathrm{Ric}_{M_i}\ge -(n-1), \quad \mathrm{vol}(B_1(p_i))\ge v>0.$$
	Let $G_i$ be a group acting isometrically on $M_i$ for each $i$. Suppose that one of the following statements holds:\\
	(1) For any sequence $f_i\in G_i$ and $r_i\le s_i\in(0,1]$ with
	$$(s^{-1}_iM_i,p_i,f_i)\overset{GH}\longrightarrow(X,p,\mathrm{id}),$$
	$$(r^{-1}_iM_i,p_i,f_i)\overset{GH}\longrightarrow(X',p',f'),$$
	if $f'$ fixes $p'$, then $f'=\mathrm{id}$, where $\mathrm{id}$ is the identity map.\\
	(2) For any sequence $f_i\in G_i$ and  $r_i\le s_i\in(0,1]$ with
	$$(s^{-1}_iM_i,p_i,f_i)\overset{GH}\longrightarrow(X,p,f),$$
	$$(r^{-1}_iM_i,p_i,f_i)\overset{GH}\longrightarrow(X',p',\mathrm{id}),$$
	then $f=\mathrm{id}$.\\
	Then there are $\epsilon,\eta>0$ such that for all $i$ and all $r\in(0,1]$, $G_i$-action has no $\epsilon$-small $\eta$-subgroup at $p_i$ with scale $r\in(0,1]$.
	
	Moreover, statements (1) and (2) are equivalent.
\end{prop}

Before presenting the proof of Proposition \ref{rescale_identity}, we make some remarks on its assumptions and connections to the scaling $\Phi$-nonvanishing property.

\begin{rem}
	In assumptions (1) and (2) of Proposition \ref{rescale_identity}, we assume that $r_i\le s_i$. The main interesting case is $r_i/s_i\to 0$. Take (1) for example, if $r_i/s_i$ subconverges to some $l\in(0,1)$, then
	$$(r_i^{-1}M_i,p_i,f_i)\overset{GH}\longrightarrow (l^{-1}X,p,f')$$
	with $f'|_{B_{l}(p)}=\mathrm{id}$. By Lemma \ref{char_id}, this means $f'=\mathrm{id}$. Thus (1) is always true if $r_i/s_i\not\to 0$. Similarly, (2) always holds if $r_i/s_i\not\to 0$.
\end{rem}

\begin{rem}\label{rem_rescale_id}
	With a standard contradicting argument, it is obvious to see the following. To verify Conjecture \ref{main_conj_vol_nv}, that is, the scaling $\Phi(\delta,n,v)$-nonvanishing property, it is equivalent to prove that (2) in Proposition \ref{rescale_identity} holds for any sequence $(M_i,p_i,f_i)$ with the curvature and volume condition. Therefore, due to Proposition \ref{rescale_identity}, scaling $\Phi$-nonvanishing property implies no $\epsilon$-small $\eta$-subgroup at $p$ for some positive constants $\epsilon(n,v,\Phi)$ and $\eta(n,v,\Phi)$. To sum up, we result in the Corollary below.
\end{rem}

\begin{cor}\label{rescale_identity_cor}
	Given $n,v>0$ and a positive function $\Phi(\delta)$, then there are positive constants $\epsilon(n,v,\Phi)$ and $\eta(n,v,\Phi)$ such that the following holds.
	
	Let $(M,p)$ be a complete $n$-manifolds with $$\mathrm{Ric}_{M}\ge -(n-1), \quad \mathrm{vol}(B_1(p))\ge v.$$
	For any isometric $G$-action on $M$, if $G$-action is scaling $\Phi$-nonvanishing at $p$, then $G$-action has no $\epsilon$-small $\eta$-subgroup at $p$ with scale $r\in(0,1]$.
\end{cor}

\begin{rem}	
	Both statements (1) and (2) in Proposition \ref{rescale_identity} would fail in general if one remove the lower volume bound. For (1), consider the sequence
	$$(S_{r_i}^2,p_i,f_i)\overset{GH}\longrightarrow (\text{point},p,\mathrm{id}),$$
	where $S_{r_i}^2$ is the round $2$-sphere of radius $r_i\to 0$ and $f_i$ is a rotation of angle $\pi$ around an axis through $p_i$ (with $s_i=1$). After rescaling $r_i^{-1}\to\infty$, we have
	$$(r_i^{-1}S_{r_i}^2,p_i,f_i)\overset{GH}\longrightarrow (S^2,p',f')$$
	with $f'$ fixing $p'$. For (2), the horn limit space \cite{CC97} we mentioned in the introduction is an example.
\end{rem}

\begin{rem}
	Conversely, for a sequence $(M_i,p_i,G_i)$ with $\mathrm{Ric}\ge -(n-1)$, if for each $i$, $G_i$-action has no $\epsilon$-small $\eta$-subgroup at $q$ with scale $r\in (0,1]$ for all $q\in B_1(p)$, then one can show that (1) and (2) in Proposition \ref{rescale_identity} holds for such a sequence.
\end{rem}

\begin{proof}[Proof of Proposition \ref{rescale_identity}]	
We first prove that (1) implies no $\epsilon$-small $\eta$-subgroup at $p_i$. We argue by contradiction. Suppose that each $G_i$ contains a symmetric subset $A_i$ with
$$\dfrac{d_{H}(A_i{p_i},A_i^2{p_i})}{\mathrm{diam}(A_i{p_i})}\to 0,$$
and $t_i^{-1}D_{t_i,p_i}(A_i)\to 0$ for some $t_i\in(0,1]$. In particular, we have convergence
$$(t_i^{-1}M_i,p_i,A_i)\overset{GH}\longrightarrow (X,p,\{e\}).$$
We choose a sequence $r_i\to 0$ as in the proof of Proposition \ref{main_nsas_local} so that for each $i$
$$r_i^{-1}D_{r_i}(A_i)=\delta/2,$$
where $\delta=\delta(n,v)$ is the constant in Theorem \ref{main_no_small_subgroup}.
By the method of choosing $r_i$, we can also assume that $$\theta^{-1}D_{\theta}(A_i)>\delta/2$$
for all $\theta<r_i$. In this way, we have $r_i\le t_i$.

We rescale the sequence by $r_i^{-1}$ as in the proof of Proposition \ref{main_nsas_local}:
$$(r_i^{-1}{M}_i,{p}_i,A_i)\overset{GH}\longrightarrow({X}',{p}',A_\infty).$$
so that $D_1(A_\infty)=\delta/2$, and thus $A_\infty$ is not a subgroup. 
At point ${p}'$, $A_\infty$-orbit satisfies
\begin{align*}
d_H(A_\infty{p}',A_\infty^2{p}')=&\lim\limits_{i\to\infty} d_H(A_i{p}_i,A^2_i{p}_i)\ \ \ (\text{on } r_i^{-1}{M}_i)\\
\le & \lim\limits_{i\to\infty} \epsilon_i\cdot \mathrm{diam}(A_i{p_i})\\
\le & \lim\limits_{i\to\infty} \epsilon_i\cdot \delta/2\to 0.
\end{align*}
This means that there is an non-identity element $a\in A_\infty^3$ fixing ${p}'$. Therefore, we have a sequence $a_i\in A_i^3$ such that
$$(t_i^{-1}{M}_i,{p}_i,a_i)\overset{GH}\longrightarrow(X,{p},\mathrm{id}),$$
$$(r_i^{-1}{M}_i,{p}_i,a_i)\overset{GH}\longrightarrow(X',{p}',a).$$
By assumptions we have $a=\mathrm{id}$, a contradiction.

\textit{Proof of (2)$\Rightarrow$(1)}. Suppose that there are $r_i\le s_i\in (0,1]$ and $f_i\in G_i$ such that
$$(s_i^{-1}M_i,p_i,f_i)\overset{GH}\longrightarrow(X,p,\mathrm{id}),$$
$$(r_i^{-1}M_i,p_i,f_i)\overset{GH}\longrightarrow(X',p',f')$$
where $f'$ fixes $p'$, but $f'\not=\mathrm{id}$. 

Without lose of generality, we assume that $f'$ has finite order. Actually, if $f'$ has infinite order, then $\overline{\langle f'\rangle}$ has a circle subgroup. We take $A_i=\{e,f_i^{\pm 1},...,f_i^{\pm k_i}\}$ such that $k_i\to\infty$ slowly and
$$(s_i^{-1}M_i,p_i,A_i)\overset{GH}\longrightarrow(X,p,\{e\}).$$
After rescaling,
$$(r_i^{-1}M_i,p_i,A_i)\overset{GH}\longrightarrow(X',p',A_\infty)$$
with $A_\infty$ containing $\overline{\langle f'\rangle}$. So there is $g_i\in A_i$ such that
$$(s_i^{-1}M_i,p_i,g_i)\overset{GH}\longrightarrow(X,p,\mathrm{id}),$$
$$(r_i^{-1}M_i,p_i,g_i)\overset{GH}\longrightarrow(X',p',g'),$$
where $g'$ fixes $p'$ and has finite order. 

We have assumed that $f'$ has finite order. Let $N<\infty$ be the order of $f'$. By Theorem \ref{main_no_small_subgroup}, on $(X',p')$ we have
$$D_1(f')\ge \delta/N.$$
By intermediate value theorem, there is an intermediate sequence $\theta_i\in (r_i,s_i)$ such that
$$\theta^{-1}_iD_{\theta_i,p_i}(f_i)=\delta/(2N).$$
Under $\theta^{-1}$, we see that
$$(\theta_i^{-1}M_i,p_i,f_i)\overset{GH}\longrightarrow(X'',p'',f''),$$
with $D_1(f'')=\delta/(2N).$ By Theorem \ref{main_no_small_subgroup}, $f''$ has order at least $2N$. Now we result in the following sequence:
$$(\theta^{-1}_iM_i,p_i,f_i^N)\overset{GH}\longrightarrow(X'',p'',(f'')^N\not=\mathrm{id});$$
$$(r^{-1}_iM_i,p_i,f_i^N)\overset{GH}\longrightarrow(X',p',(f')^N=\mathrm{id}).$$
This contradicts the assumption.

\textit{Proof of (1)$\Rightarrow$(2)}. The proof is very similar to the one of \textit{(2)$\Rightarrow$(1)}. If the statement is false, then one can find a contradiction to (1) in some intermediate rescaling sequence.
\end{proof}

Recall that to verify Conjecture \ref{main_conj_vol_nv}, it is enough to deal with the case $s=1$ in Definition \ref{main_def_nonvanish} due to relative volume comparison. As seen in Remark \ref{rescale_identity_cor}, it suffices to rule out a sequence $(M_i,p_i,f_i)$ with 
$$\mathrm{Ric}_{M_i}\ge -(n-1), \quad \mathrm{vol}(B_1(p_i))\ge v>0$$
and its rescaling sequence $(r^{-1}_i\to\infty)$ with:
$$(M_i,p_i,f_i)\overset{GH}\longrightarrow (X,p,f\not=\mathrm{id})$$
$$(r_i^{-1}M_i,p_i,f_i)\overset{GH}\longrightarrow (X',p',\mathrm{id}).$$
We can further reduce the above sequence to the following situation:

\textit{Without lose of generality, we can assume that $f$ has finite order and both $X$, $X'$ are metric cones.}

In fact, if $f$ has infinite order, then we consider a sequence of symmetric subsets $A_i=\{e,f_i^{\pm 1},...,f_i^{\pm k_i}\}$. We choose $k_i\to\infty$ slowly so that
$$(r^{-1}_iM_i,p_i,A_i)\overset{GH}\longrightarrow(X',p',\{e\}).$$
Since before rescaling, the limit of $f_i$ fixes $p$. Thus the limit of $A_i$ contains a circle subgroup fixing $p$. As a result, there is $g_i\in A_i$ such that
$$(M_i,p_i,g_i)\overset{GH}\longrightarrow(X,p,g),$$
$$(r_iM_i,p_i,g_i)\overset{GH}\longrightarrow(X',p',\mathrm{id}),$$
where $g$ fixes $p$ and has finite order.

Reduction to metric cones follows directly from the lemma below and a standard rescaling argument by passing to tangent cones (see Theorem \ref{CC_cone}). More precisely, under the conditions of Proposition \ref{rescale_identity}, we can find $s_i\to\infty$, $s'_i\to\infty$ with $s'_i/s_i\to \infty$ and
$$(s_iM_i,p_i)\overset{GH}\longrightarrow(C_pX,o)$$
$$(s'_iM_i,p_i)\overset{GH}\longrightarrow(C_{p'}X',o').$$

\begin{lem}
	Let $(Y,p)$ be an non-collapsing Ricci limit space and $f$ be any isometry of $Y$ fixing $p$. Suppose that $f$ has finite order $k$, then for any $r_i\to\infty$ and any convergent subsequence
	$$(r_iY,p,f)\overset{GH}\longrightarrow(C_pY,o,f_p),$$
	$f_p$ has order $k$.
\end{lem}

\begin{proof}
    Because $f$ has finite order, for any $r_i\to \infty$ and any  convergent subsequence, we have
	$$(r_iY,p,\langle f\rangle)\overset{GH}\longrightarrow(C_pY,o,\langle f_p\rangle).$$
	Since $f$ has order $k$, $f_p$ has order at most $k$. Suppose that $f_p$ has order $l<k$. This implies that $$(r_iY,p,f^l)\overset{GH}\longrightarrow(C_pY,o,e).$$
	Together with the fact that $\langle f\rangle$ is a discrete group, we see that
	$$(r_iY,p,\langle f^l \rangle )\overset{GH}\longrightarrow(C_pY,o,\{e\}).$$
    By Theorem \ref{main_no_small_subgroup}, $\langle f^l\rangle=e$, a contradiction.
\end{proof}

Next we show that the scaling $\Phi$-nonvanishing property holds when $\mathrm{sec}_{M_i}\ge -1$ (volume condition is not required in this situation). As pointed out before, it suffices to prove the lemma below on sequences.

\begin{lem}\label{rescale_identity_sec}
	Let $(M_i,p_i)$ be a sequence of $n$-manifolds with
	$\mathrm{sec}_{M_i}\ge-1$ and $f_i$ be a sequence of isometries of $M_i$. Suppose that $r_i^{-1}\to\infty$ with
	$$(M_i,p_i,f_i)\overset{GH}\longrightarrow(X,p,f),$$
	$$(r^{-1}_iM_i,p_i,f_i)\overset{GH}\longrightarrow(X',p',\mathrm{id}).$$
	Then $f=\mathrm{id}$.
\end{lem}

For $0<r\le R$, we define the $(r,R)$-scale segment domain at $p$ as follows.
$$S^R_r(p)=\{\gamma|_{[0,r]}\ |\  \gamma \text{ is a unit speed minimal geodesic from $p$ of length at least $R$}\}.$$
Note that $S^R_r(p)$ is always a subset of $B_r(p)$, but it may not be equal to $B_r(p)$. We also define the $r$-scale exponential map at $p$ ($0<r<1$):
\begin{align*}
	\exp_p^r: &\  S^1_r(p)\ \ \to B_1(p)\\
	&\exp_p(v)\mapsto \exp_p(r^{-1}v).
\end{align*}

\begin{lem}\label{converge_cone}
	If $(M_i,p_i)\overset{GH}\rightarrow (X,p)$ and $(X,p)$ is a metric cone with vertex $p$, then $S^1_1(p_i)\overset{GH}\rightarrow B_1(p)$.
\end{lem}

\begin{proof}
	For any $z\in B_1(p)$ with $z\not=p$, put $d=d(z,p)$. Let $\gamma$ be the unique unit speed minimal geodesic from $p$ to $z$. Extend $\gamma$ to a ray starting at $p$ and put $q:=\gamma(2)$. Pick $q_i\in M_i$ with $q_i\to q$. For each $i$, let $\gamma_i$ be a unit speed minimal geodesic from $p_i$ to $q_i$. It is clear that the image of $\gamma_i|_{[0,1]}$ is in $S^1_1(p_i)$. $\gamma_i$ converges to a minimal geodesic from $p$ to $q$, which must be $\gamma|_{[0,2]}$. In particular, $\gamma_i(d)\to z$.
\end{proof}

\begin{proof}[Proof of Lemma \ref{rescale_identity_sec}]
	As discussed above on the reduction, we may assume that both $X$ and $X'$ are metric cones (Note that both $X$ and $X'$ are Alexandrov spaces, thus their tangent cones are always metric cones \cite{BGP92}).
	
	For each $i$, we consider the commutative diagram:
	\begin{center}
		$\begin{CD}
		r_i^{-1}S^1_{r_i}(p_i) @>f_i>> r_i^{-1}S^1_{r_i}(f_i(p_i))\\
		@VVr_i^{-1}\exp_{p_i}^{r_i}V @VVr_i^{-1}\exp_{f_i(p_i)}^{r_i} V\\
		B_1(p_i) @>f_i>> B_1(f_i(p_i))
		\end{CD}$
	\end{center}
	
	Let $S(p')\subseteq B_1(p')$ be the Gromov-Hausdorff limit of $r^{-1}_i S^1_{r_i}(p_i)$. Since
	$$(r^{-1}_iM_i,p_i,f_i)\overset{GH}\longrightarrow(X',p',\mathrm{id}),$$
	$S(p')$ is also the limit of $r_i^{-1} S^1_{r_i}(f_i(p_i))$. By Toponogov theorem, both $r_i^{-1}\exp_{p_i}^{r_i}$ and $r_i^{-1}\exp_{f_i(p_i)}^{r_i}$ are $L(n)$-Lipschitz maps. Passing to a subsequence, these two sequences of maps converge to $\alpha$ and $\alpha':S(p')\to B_1(p)$ as $i\to\infty$ respectively. By Lemma \ref{converge_cone}, $\alpha$ and $\alpha'$ are surjective. We claim that $\alpha=\alpha'$. In fact, if for some $q\in S(p')$, $\alpha(q)\not=\alpha'(q)$, then we can find minimal geodesics $\gamma_i$ and $\gamma_i'$ from $p_i$ such that
	$$(r_i^{-1}M_i,\gamma_i(r_id),\gamma'_i(r_id))\overset{GH}\longrightarrow(X,q,q)$$
	$$(M_i,\gamma_i(d),\gamma'_i(d))\overset{GH}\longrightarrow(X,\alpha(q),\alpha'(q)),$$
	where $d=d(p,q)$. By Toponogov theorem, we see a bifurcation of minimal geodesics at $q$, but we know this cannot happen in $X'$ \cite{BGP92}.
	
	Now we have a commutative diagram of limit spaces
	\begin{center}
		$\begin{CD}
		S(p') @>\mathrm{id}>> S(p')\\
		@VV\alpha V @VV\alpha V\\
		B_1(p) @>f>> B_1(p),
		\end{CD}$
	\end{center}
	where $f$ is an isometry and $\alpha$ is surjective. Therefore, $f=\mathrm{id}$.
\end{proof}

\begin{cor}\label{nsas_sec}
	Given $n$, there is a positive function $\Phi(\delta,n)$ such that for any complete $n$-manifold $(M,p)$ of
	$\mathrm{sec}\ge -1$, any isometry of $M$ is scaling $\Phi(\delta,n)$-nonvanishing at $p$.
\end{cor}

\subsection{Equivariant stability}

As an application of Theorem \ref{main_no_small_subgroup}, we prove the following stability result, which implies finiteness of fundamental groups in \cite{An90a}.

\begin{thm}\label{main_stable}
	Let $(M_i,p_i)$ be a sequence of closed $n$-manifolds with
	$$\mathrm{Ric}_{M_i}\ge -(n-1),\quad \mathrm{diam}(M)\le D,\quad \mathrm{vol}(B_1(p_i))\ge v>0$$
	If the following sequences converge in the Gromov-Hausdorff topology
	\begin{center}
		$\begin{CD}
		(\widetilde{M}_i,\tilde{p}_i,\Gamma_i) @>GH>> (\widetilde{X},\tilde{p},G)\\
		@VV\pi_iV @VV\pi V\\
		(M_i,{p}_i) @>GH>> (X,p),
		\end{CD}$
	\end{center}
	then $\Gamma_i$ is isomorphic to $G$ for all $i$ large.
\end{thm}

Recall that Theorem \ref{main_no_small_subgroup} implies that if $(M,p)$ satisfies
$$\mathrm{Ric}\ge -(n-1),\quad\mathrm{vol}(B_1(\tilde{p}))\ge v>0,$$
then any nontrivial subgroup $H$ of $\Gamma$ has $D_{1,\tilde{p}}(H)\ge \delta(n,v)$. Under a stronger volume condition $$\mathrm{vol}(B_1(p))\ge v>0,$$ we show that such a lower bound on displacement holds for any nontrivial covering transformation.

\begin{lem}\label{no_small_ismoetry}
	Given $n$ and $v>0$, there is a constant $\delta(n,v)>0$ such that for any $n$-manifold $(M,p)$ with
	$$\mathrm{Ric}\ge -(n-1),\quad \mathrm{vol}(B_1(p))\ge v$$
	and any nontrivial element $\gamma\in \pi_1(M,p)$, we have
	$D_{1,\tilde{p}}(\gamma)\ge \delta$.
\end{lem}

\begin{proof}
	We argue by contradiction. Suppose that we have the following convergent sequences
	\begin{center}
		$\begin{CD}
		(\widetilde{M}_i,\tilde{p}_i,\Gamma_i) @>GH>> (\widetilde{X},\tilde{p},G)\\
		@VV\pi_iV @VV\pi V\\
		(M_i,{p}_i) @>GH>> (X,p).
		\end{CD}$
	\end{center}
	with
	$$\mathrm{Ric}\ge -(n-1),\quad \mathrm{vol}(B_1(p))\ge v;$$
	and a sequence of nontrivial elements $\gamma_i\in \Gamma_i$ converging to the identity map, where $\Gamma_i=\pi_1(M_i,p_i)$.
	
	By \cite{An90a}, there are positive constants $L(n,v)$ and $N(n,v)$ such that for any subgroup in $\pi_1(M,p)$ generated by elements of length $\le L$, this subgroup has order $\le N$ (In \cite{An90a}, only closed manifolds with bounded diameter are considered, but its proof extends to open manifolds).  Since $\gamma_i\to \mathrm{id}$, for all $i$ large $\gamma_i$ has length $\le L$, thus has order $\le N$. Consequently, the sequence of subgroups $\langle \gamma_i\rangle$ also converges to $\{e\}$. By Theorem \ref{main_no_small_subgroup}, this implies that $\langle \gamma_i\rangle$, and thus $\gamma_i$, is identity for $i$ large.
\end{proof}

With Lemma \ref{no_small_ismoetry}, we prove Theorem \ref{main_stable}, the stability of $\pi_1$ under equivariant GH convergence for non-collapsing manifolds with bounded diameter.

\begin{proof}[Proof of Theorem \ref{main_stable}]	
	We first notice that $G$ is a discrete group (intuitively, otherwise $M_i$ would be collapsed). In fact, we consider $\langle\Gamma_i(L)\rangle$, the subgroup generated by loops of length $\le L$, where $L=L(n,v)$ is the constant mentioned in the proof of Lemma \ref{no_small_ismoetry}. We consider
	$$(\widetilde{M}_i,\tilde{p}_i,\Gamma_i(L))\overset{GH}\longrightarrow(\widetilde{X},\tilde{p},H).$$
	Since each $\Gamma_i(L)$ has order $\le N$, so does $H$. Note that $H$ contains $G_0$, thus $G_0=\{e\}$ and $G$ is discrete.
	
    By \cite{FY92}, there exists a sequence of subgroups $H_i$ of $\Gamma_i$ such that
    $$(\widetilde{M}_i,\tilde{p}_i,H_i)\overset{GH}\longrightarrow(\widetilde{X},\tilde{p},G_0)$$
    and $\Gamma_i/H_i$ is isomorphic to $G/G_0$ for all $i$ large. In our situation, $G_0=\{e\}$ and thus $H_i\overset{GH}\to\{e\}$. By Theorem \ref{no_small_ismoetry}, we see that $H_i=\{e\}$ for all $i$ large. Consequently, $\Gamma_i$ is isomorphic to $G$ for all $i$ large.
\end{proof}

\subsection{$C$-abelian of fundamental groups}

We prove two structure theorems below on fundamental groups of closed manifolds.

\begin{thm}\label{main_abel_1}
	Given $n,D,v>0$, there exists a constant $C(n,D,v)$ such that if a complete $n$-manifold $(M,p)$ with finite fundamental group satisfies
	$$\mathrm{Ric}\ge -(n-1), \quad \mathrm{diam}(M)\le D,\quad \mathrm{vol}(B_1(\tilde{p}))\ge v>0,$$
	then $\pi_1(M)$ contains an abelian subgroup of index $\le C(n,v)$. Moreover, this subgroup can be generated by at most $n$ elements.
\end{thm}

\begin{thm}\label{main_abel_2}
	Given $n,v>0$, there exists a constant $C(n,v)$ such that if a complete $n$-manifold $(M,p)$ satisfies
	$$\mathrm{Ric}\ge 0,\quad \mathrm{diam}(M)=1,\quad \mathrm{vol}(B_1(\tilde{p}))\ge v>0,$$
	then $\pi_1(M)$ contains an abelian subgroup of index $\le C(n,v)$. Moreover, this subgroup can be generated by at most $n$ elements.
\end{thm}

Theorems \ref{main_abel_1} and \ref{main_abel_2} generalize Theorems D and E in \cite{MRW08}, where the curvature conditions are on sectional curvature. Given Theorem 8 in \cite{KW11} and Theorem 4.1 \cite{CC00a}, actually their proof \cite{MRW08} extends to the Ricci case. Here we give an alternative approach by applying Theorem \ref{main_no_small_subgroup} and Kapovitch-Wilking's work \cite{KW11}.

Theorems \ref{main_abel_1} and \ref{main_abel_2} partially verify the following conjectures respectively.

\begin{conj}
	Given $n$ and $D$, there exists a constant $C(n,D)$ such that the following holds. Let $M$ be an $n$-manifold with finite fundamental group and
	$$\mathrm{Ric}\ge -(n-1),\quad\mathrm{diam}(M)\le D,$$
	then $\pi_1(M)$ contains an abelian subgroup of index $\le C(n,D)$. Moreover, this subgroup can be generated by at most $n$ elements.
\end{conj}

\begin{conj}[Fukaya-Yamaguchi]
	Given $n$, there exists a constant $C(n)$ such that for any $n$-manifold with nonnegative Ricci curvature, its fundamental group $\pi_1(M)$ contains an abelian subgroup of index $\le C(n)$. Moreover, this subgroup can be generated by at most $n$ elements.
\end{conj}

We make use of the following result on nilpotent groups.


\begin{lem}\label{commutator_length}
	\cite{St66} Let $\Gamma$ be a nilpotent group generated by $n$ elements $x_1,...,x_n$. Then every element in $[\Gamma,\Gamma]$ is a product of $n$ commutators $[x_1,g_1],...,[x_n,g_n]$ for suitable $g_i\in G$ $(i=1,...,n)$.
\end{lem}

\begin{proof}[Proof of Theorem \ref{main_abel_1}]
	Suppose that the statement does not hold, then we have a contradicting sequence
	\begin{center}
		$\begin{CD}
		(\widetilde{M}_i,\tilde{p}_i,\Gamma_i) @>GH>> (\widetilde{X},\tilde{p},G)\\
		@VV\pi_iV @VV\pi V\\
		(M_i,p_i) @>GH>> (X,p)
		\end{CD}$
	\end{center}
	with finite fundamental groups and
	$$\mathrm{Ric}_{M_i}\ge -(n-1),\quad\mathrm{diam}(M_i)=D, \quad \mathrm{vol}(B_1(\tilde{p}_i))\ge v>0,$$
	but any abelian subgroup in $\pi_1(M_i)$ has index larger than $i$. By \cite{KW11}, $\Gamma_i$ is $C(n)$-nilpotent with a cyclic chain of length $\le n$. Thus without lose of generality, we may assume that $\Gamma_i$ is nilpotent with a cyclic chain of length $\le n$ for all $i$, and thus $G$ is a nilpotent Lie group.
	
	By Diameter Ratio Theorem \cite{KW11}, $\mathrm{diam}(\widetilde{M}_i)$ has an upper bound $\widetilde{D}(n,D)$. Thus the limit space $\widetilde{X}$ and its limit group $G$ are compact. $G_0$, as a connected compact nilpotent Lie group, must be a torus. We call this torus $T$. Since $G$ is compact, there is a sequence of subgroups $H_i$ converging to $T$ such that 
	$$[\Gamma_i:H_i]=[G:T]<\infty.$$
	We complete the proof once we show that $H_i$ is abelian and can be generated by at most $n$-elements.
	
	Since $\Gamma_i$ is nilpotent with a cyclic chain of length $\le n$, $H_i$ can be generated by at most $n$-elements. To show that $\Gamma_i$ is abelian, we consider $[H_i,H_i]$, the subgroup of $H_i$ generated by all commutators. We claim that $[H_i,H_i]\overset{GH}\longrightarrow e$, then by Corollary \ref{stable_nss}, $[H_i,H_i]=e$ and thus $H_i$ is abelian. Indeed, for any sequence $\gamma_i$ in $[H_i,H_i]$, by lemma \ref{commutator_length} it can be written as $\prod_{j=1}^n [x_{i,j},h_{i,j}]$, where $\{x_{i,j}\}_{j=1}^n$ are generators of $H_i$ and $h_{i,j}\in H_i$. Since the limit group $T$ is compact, passing to a subsequence if necessary, we may assume that $x_{i,j}\to x_j\in T$ and $h_{i,j}\to h_j\in T$. Because $T$ is abelian, $[x_{i,j},h_{i,j}]\to [x_j,h_j]=e$ and thus $\gamma_i\to e$.
\end{proof}

Next we consider closed manifolds with nonnegative Ricci curvature.

\begin{lem}\label{diameter_ratio_2}
	Given $n$, there exists a constant $C(n)$ such that the following holds. Let $M$ be a closed $n$-Riemannian manifold with
	$$\mathrm{Ric}\ge 0,\quad\mathrm{diam}(M)=1.$$
	Then $\widetilde{M}$ splits isometrically as $N\times \mathbb{R}^k$ with $\mathrm{diam}(N)\le C(n)$.
\end{lem}

\begin{proof}
	By Cheeger-Gromoll splitting theorem \cite{CG72a}, we know that $\widetilde{M}$ splits isometrically as $N\times \mathbb{R}^k$, where $N$ is compact and simply connected. Suppose that we have a contradicting sequence: $M_i$ with
	$$\mathrm{Ric}_{M_i}\ge 0,\quad\mathrm{diam}(M_i)=1,$$
	but $N_i$, the compact factor of $\widetilde{M_i}$, has diameter $\to\infty$. By generalized Margulis Lemma \cite{KW11}, it is easy to see that $\Gamma_i=\pi_1(M_i,p_i)$ is $C(n)$-nilpotent. Hence without lose of generality, we may assume that $\Gamma_i$ itself is nilpotent.
	
	Put $r_i=\mathrm{diam}(N_i)\to\infty$ and consider the rescaling sequence
	\begin{center}
		$\begin{CD}
		(r_i^{-1}N_i\times\mathbb{R}^k,\tilde{p}_i,\Gamma_i) @>GH>> (Y\times\mathbb{R}^k,\tilde{p},G)\\
		@VV\pi_iV @VV\pi V\\
		(r_i^{-1}M_i,p_i) @>GH>> point
		\end{CD}$
	\end{center}
	where $G$ is a nilpotent Lie group acting transitively on the limit space $Y\times\mathbb{R}^k$. Let $K$ be the subgroup of $G$ acting trivially on $\mathbb{R}^k$-factor. Then $K$ acts effectively and transitively on $Y$. In particular, $Y$ is a compact topological manifold homeomorphic to $K/\mathrm{Iso}$, where $\mathrm{Iso}$ is the isotropy subgroup of $K$. Note that $K_0$ is connected, compact, and nilpotent; thus $K_0$ is a torus, which acts transitively and effectively on $Y$. With these facts, it is easy to verify that $Y$ itself is also a torus.
	
	On the other hand, we have $r_i^{-1}N_i\overset{GH}\longrightarrow Y$. Since each $N_i$ is simply connected and $Y$ is a compact topological manifold, $Y$ must be simply connected as well. We end in a contradiction.
\end{proof}

\begin{rem}
	We point out that in \cite{MRW08} the proof of Theorem D, the diameter bound $\mathrm{diam}(N)\le C(n)$ is asserted by an incorrect inequality.
\end{rem}

\begin{proof}[Proof of Theorem \ref{main_abel_2}]
	We argue by contradiction. Suppose the contrary, then we have a contradicting sequence $M_i$ with
	$$\mathrm{Ric}_{M_i}\ge 0,\quad\mathrm{diam}(M_i)=1,\quad\mathrm{vol}(B_1(\tilde{p_i}))\ge v>0,$$
	but any abelian subgroup of $\pi_1(M_i)$ has index $>i$. By generalized Margulis Lemma \cite{KW11}, we may assume that for each $i$, $\pi_1(M_i)$ is nilpotent with a cyclic chain of length at most $n$.
	
	By Lemma \ref{diameter_ratio_2}, $\widetilde{M_i}$ splits as $N_i\times \mathbb{R}^{k_i}$ isometrically with $\mathrm{diam}(N_i)\le C(n)$. Since $k_i\le n$ for all $i$, passing to a subsequence, we may assume $k_i=k$ for all $i$. Passing to a subsequence again, we obtain the following convergent sequences.
	\begin{center}
		$\begin{CD}
		(N_i\times \mathbb{R}^{k},\tilde{p}_i) @>GH>> (N\times\mathbb{R}^{k},\tilde{p})\\
		@VVV @VVV\\
		(M_i,p_i) @>GH>> (X,p),
		\end{CD}$
	\end{center}
	where $N$ is compact. From the assumption that $\mathrm{vol}(B_1(\tilde{p_i}))\ge v>0$, it is obvious that $\mathrm{vol}(N_i)\ge v_0>0$ for some $v_0$.
	
	Let $p_i:\mathrm{Isom}(N_i\times\mathbb{R}^k)\to\mathrm{Isom}(\mathbb{R}^k)$ and $q_i:\mathrm{Isom}(N_i\times\mathbb{R}^k)\to\mathrm{Isom}(N_i)$
	be the natural projection maps. Consider $\overline{q_i(\Gamma_i)}$ acting on $N_i$ and the corresponding convergent sequence
	$$(N_i,\overline{q_i(\Gamma_i)})\overset{GH}\longrightarrow (N,G).$$
	$N$ is compact and thus $G$ is also compact. Then by a similar argument in the proof of Theorem \ref{main_abel_1}, we can show that $\overline{q_i(\Gamma_i)}$, and thus $q_i(\Gamma_i)$, is $C_1$-abelian, where $C_1$ is a constant independent of $i$. Also, ${p}_i(\Gamma_i)$ acts co-compactly on $\mathbb{R}^k$, thus by Bieberbach theorem, ${p}_i(\Gamma_i)$ is $C_2(n)$-abelian.
	
	Finally, we treat $\Gamma_i$ as a subgroup of ${q_i(\Gamma_i)}\times {p}_i(\Gamma_i)$. It is easy to check that $\Gamma_i$ contains an abelian subgroup of index $\le C_1C_2$. Moreover, this subgroup can be generated by at most $n$-elements because $\Gamma_i$ is nilpotent with a cyclic chain of length at most $n$.
\end{proof}

\section{Dimension monotonicity of symmetries}

We prove our main technical result, dimension monotonicity of symmetries. For a space $(Y,q,H)$, we always assume that $(Y,q)\in\mathcal{M}(n,-1,v)$ and $H$ is a closed abelian subgroup of $\mathrm{Isom}(Y)$; in particular, $H$-action is always effective. We state the dimension monotonicity of symmetries as follows.

\begin{thm}\label{dimension}
	Let $(M_i,p_i)$ be a sequence of complete $n$-manifolds with $$\mathrm{Ric}_{M_i}\ge -(n-1),\quad \mathrm{vol}(B_1(p_i))\ge v>0;$$
	let $\Gamma_i$ be a closed abelian subgroup of $\mathrm{Isom}(M_i)$ for each $i$. Suppose that there is a positive function $\Phi$ such that $\Gamma_i$-action is scaling $\Phi$-nonvanishing at $p_i$ for all $i$.\\
	If the following two sequences converge $(r_i\to\infty)$:
	$$({M}_i,{p}_i,\Gamma_i)\overset{GH}\longrightarrow ({X},{p},G),$$
	$$(r_i{M}_i,{p}_i,\Gamma_i)\overset{GH}\longrightarrow ({X}',{p}',G'),$$
	then the following holds:\\
	(1) $\dim(G')\le\dim(G)$;\\
	(2) If $G'$ has a compact subgroup $K'$, then $G$ contains a subgroup $K$ fixing ${p}$ and $K$ is isomorphic to $K'$.
\end{thm}

\begin{rem}\label{rm_abel_nil}
	We expect that Theorem \ref{dimension} holds for nilpotent group actions with controlled nilpotency length, which is enough to remove the abelian assumption in Theorem \ref{main_sg}. Generalizing Theorem \ref{dimension} to the nilpotent case requires much more work.
\end{rem}

For convenience, we reformulate Proposition \ref{rescale_identity} and Corollary \ref{rescale_identity_cor} here.

\begin{prop}\label{non_van_prop}
	Let $(M_i,p_i,\Gamma_i)$ be a sequence with the assumptions in Theorem \ref{dimension}. Then the following holds:\\
	(1) For any sequence $f_i\in \Gamma_i$ and $r_i\ge s_i\ge 1$ with
	$$(s_iM_i,p_i,f_i)\overset{GH}\longrightarrow (Y,p,\mathrm{id}),$$
	$$(r_iM_i,p_i,f_i)\overset{GH}\longrightarrow (Y',p',f'),$$
	if $f'$ fixes $p'$, then $f'=\mathrm{id}$.\\
	(2) For any sequence $f_i\in \Gamma_i$ and $r_i\ge s_i\ge 1$ with
	$$(s_iM_i,p_i,f_i)\overset{GH}\longrightarrow (Y,p,f),$$
	$$(r_iM_i,p_i,f_i)\overset{GH}\longrightarrow (Y',p',\mathrm{id}),$$
	then $f=\mathrm{id}$.\\
	(3) There are positive constants $\epsilon(n,v,\Phi)$ and $\eta(n,v,\Phi)$ such that $\Gamma_i$-action has no $\epsilon$-small $\eta$-subgroup at $p_i$ with scale $r\in(0,1]$ for each $i$.
\end{prop}

We will later see that the no small almost subgroup property is the key criterion for dimension monotonicity of symmetries. One can even replace volume and scaling nonvanishing assumption by a no small almost subgroup assumption around $p_i$:

\begin{thm}\label{dimension'}
	Let $(M_i,p_i)$ be a sequence of complete $n$-manifolds with $$\mathrm{Ric}_{M_i}\ge -(n-1)$$
	and let $\Gamma_i$ be a closed abelian subgroup of $\mathrm{Isom}(M_i)$ for each $i$. Suppose that there are $\epsilon,\eta>0$ such that $\Gamma_i$-action has no $\epsilon$-small $\eta$-subgroup at $q$ with scale $r\in(0,1]$ for all $q\in B_1(p)$ and for all $i$. If the following two sequences converge $(r_i\to\infty)$:
	$$({M}_i,{p}_i,\Gamma_i)\overset{GH}\longrightarrow ({X},{p},G),$$
	$$(r_i{M}_i,{p}_i,\Gamma_i)\overset{GH}\longrightarrow ({X}',{p}',G'),$$
	then the following holds:\\
	(1) $\dim(G')\le\dim(G)$;\\
	(2) If $G'$ has a compact subgroup $K'$, then $G$ contains a subgroup $K$ fixing ${p}$ and $K$ is isomorphic to $K'$.
\end{thm}

The proof of Theorem \ref{dimension'} is a mild modification of the proof of Theorem \ref{dimension}. For our purpose, we only focus on Theorem \ref{dimension} in this paper. To illustrate the rule of no small almost subgroup in the proof of Theorem \ref{dimension}, we consider the following examples.

\begin{exmps}\label{ex_1}
	Let $M_i=\mathbb{R}\times (S^3, \frac{1}{i}d_0)$, where $d_0$ is the standard metric on $S^3$, and $p_i$ be a point in $M_i$. $S^3$ admits a circle group $S^1$ acting freely and isometrically on $S^3$. For a number $\theta\in S^1=[0,2\pi]/\sim$, we denote $R(\theta)$ as the corresponding isometry on $S^3$. We define two isometries of $M_i$ by
	\begin{align*}
		\alpha_i(x,y)=&(x+i^{-2},R(2\pi/i)y);\\
		\beta_i(x,y)=&(x+i^{-3},R(2\pi/i)y).
	\end{align*}
	As $i\to\infty$, both $\langle\alpha_i\rangle$-action and $\langle\beta_i\rangle$-action converges to standard $\mathbb{R}$-translations in the limit space $\mathbb{R}$, because $S^3$-factor disappears in the limit. Now we rescale this sequence by $r_i=i$. Then $r_iM_i=\mathbb{R}\times (S^3,g_0)$, on which $\alpha_i$ and $\beta_i$ acts as
	\begin{align*}
		\alpha_i(x,y)=&(x+i^{-1},R(2\pi/i)y);\\
		\beta_i(x,y)=&(x+i^{-2},R(2\pi/i)y).
	\end{align*}
	It is clear that
	$$(r_iM_i,p_i,\langle\alpha_i\rangle,\langle\beta_i\rangle)\overset{GH}\longrightarrow(\mathbb{R}\times S^3,p',\mathbb{R},\mathbb{R}\times S^1).$$
	The limit group of $\langle\alpha_i\rangle$ is $\mathbb{R}$ acting as
	$$t\cdot(x,y)=(x+t,R(2\pi t)y), \ t\in\mathbb{R},$$
	while the limit group of $\langle\beta_i\rangle$ has an extra dimension. This extra dimension comes from a sequence of collapsed almost subgroups in $\langle\beta_i\rangle$. More precisely, if we put $B_i=\{e,\beta_i^{\pm 1},...,\beta_i^{\pm (i-1)}\}$, then on $(M_i,p_i)$ we have $D_{1,p_i}(B_i)\to 0$ and
	$$\dfrac{d_H(B_i{p_i},B_i^2{p_i})}{\mathrm{diam}(B_i{p_i})}\to 0.$$
	On $(M_i,p_i,\langle\alpha_i\rangle)$, there is no such small almost subgroup. We can take the same symmetric subsets $A_i=\{e,\alpha_i^{\pm 1},...,\alpha_i^{\pm (i-1)}\}$. Although $d_H(A_i{p_i},A_i^2{p_i})\to 0$ and $D_{1,p_i}(A_i)\to 0$, the above ratio is away from $0$ for all $i$
	$$\dfrac{d_H(A_i{p_i},A_i^2{p_i})}{\mathrm{diam}(A_i{p_i})}\ge 1/2\pi.$$
\end{exmps}

The proof of Theorem \ref{dimension} is technical and involved. We have illustrated on how to rule out $G=\mathbb{R}$ with $G'=\mathbb{R}\times S^1$ in the introduction. Here we give some indications on how to rule out $G=\mathbb{R}$ with $G'=\mathbb{R}^2$. Suppose that $G'$-action is standard translation for simplicity. One may consider a parameter $s$ changing the scale from $1$ to $r_i$ as $1+s(r_i-1)$, $s\in[0,1]$. In this way, one may imagine that there is a path, consisting of intermediate rescaling limits, and varying from $\mathbb{R}$-action to $\mathbb{R}^2$-translation. Then we can find an intermediate rescaling sequence $s_i\to\infty$ with $r_i/s_i\to\infty$ and
$$(s_i{M}_i,{p}_i,\Gamma_i)\overset{GH}\longrightarrow(Y,q,H),$$
where $H$-action is very close to $\mathbb{R}^2$-translation in the equivariant Gromov-Hausdorff topology but $H\not=\mathbb{R}^2$. If $H=\mathbb{R}\times\mathbb{Z}$, then we can apply a scaling trick to rule it out (see proof of Proposition \ref{dim_free}(1) for details). If $H=\mathbb{R}\times S^1$, then we result in the case that we know cannot happen. The situation that needs some additional arguments is $H=\mathbb{R}$, whose action is very close to $\mathbb{R}^2$-translation. We take a closer look at such an $\mathbb{R}$-action.
\begin{exmp}\label{ex_2}
	Consider $M_i=\mathbb{R}\times (S^1, i\cdot d_0)$ and $\mathbb{R}$ acting on $M_i$ by
	$$t(x,y)=(x+t/i,R(2\pi t)y),\ \  t\in\mathbb{R}.$$
	Then $$d_{GH}((M_i,p_i,\mathbb{R}),(\mathbb{R}^2,0,\mathbb{R}^2))\le 2/i,$$
	where $d_{GH}$ means the pointed equivariant Gromov-Hausdorff distance.
\end{exmp}
Note that in this particular example, $\mathbb{R}$-action on $M_i$ has almost subgroups. For $A=[-1,1]\subseteq\mathbb{R}$, we have
$$\dfrac{d_H(A{p_i},A^2{p_i})}{\mathrm{diam}(A{p_i})}\le 1/(2\pi i^2).$$
A key observation is that such phenomenon also happens in the general case: if a $\mathbb{R}$-action is very close to some $\mathbb{R}^2$-action, then it must contain an almost subgroup (see Lemma \ref{dim_free_gap}). This observation is the key to rule out such an intermediate rescaling sequence.

We start with some definitions.

\begin{defn}\label{1-para_sym_subset}
	Let $G$ be a Lie group. We say that a symmetric subset $A$ of $G$ is one-parameter, if $A$ has one of the following forms:\\
	\upperRomannumeral{1}. $A=\{e,g^{\pm 1},...,g^{\pm k}\}$ for some $g\in G$ and $k\in\mathbb{Z}^+$;\\
	\upperRomannumeral{2}. $A=\{\exp(tv)\ |\ t\in [-1,1]\}$ for some $v\in\mathfrak{g}$, the Lie algebra of $G$.
\end{defn}

\begin{defn}\label{untwisted_action}
	Let $\eta>0$ and $(Y,q,G)$ be a space. We say that $G$-action has no $\eta$-subgroup of one-parameter at $p\in Y$, if for any one-parameter symmetric subset $A\subseteq G$ with $\mathrm{diam}(Ap)\in(0,\infty)$, we have
	$$\dfrac{d_H(A{p},A^2{p})}{\mathrm{diam}(A{p})}\ge\eta.$$
\end{defn}

\begin{lem}\label{1-para_sym_subset_form}
	Let $(Y,q,\mathbb{R})$ be a space. Then the followings are equivalent:\\
	(1) $\mathbb{R}$ contains a one-parameter symmetric subset $A$ of form \upperRomannumeral{1} with
	$$\dfrac{d_H(A{q},A^2{q})}{\mathrm{diam}(A{q})}<\eta.$$
	(2) $\mathbb{R}$ contains a one-parameter symmetric subset $B$ of form \upperRomannumeral{2} with
	$$\dfrac{d_H(B{q},B^2{q})}{\mathrm{diam}(B{q})}<\eta.$$
\end{lem}

\begin{proof}
	Suppose that $\mathbb{R}$ contains a one-parameter symmetric subset $A$ of form \upperRomannumeral{1} with
	$$\dfrac{d_H(A{q},A^2{q})}{\mathrm{diam}(A{q})}<\eta.$$
	We write $A$ as $\{e,g^{\pm 1},...,g^{\pm k}\}$. For $g^{2k}\in A^2$, there is $g^n\in A$ with
	$$d(g^{2k}q,g^n q)<\eta\cdot\mathrm{diam}(Aq).$$
	
	\noindent\textit{Case 1:} $n\ge0$.
	
	Since $g\in \mathbb{R}$, $g=\exp(v)$ for some $v\in\mathfrak{g}=\mathbb{R}$. Consider $$B=\{\exp(tv)\ |\ t\in[-k,k]\}.$$
	For any $s\in[0,k]$,
	$$d(\exp((2k-s)v)q,\exp((n-s)v)q)<\eta\cdot\mathrm{diam}(Aq)\le\eta\cdot\mathrm{diam}(Bq)$$
	with $\exp((n-s)v)\in B$. Thus $B$ is a one-parameter symmetric subset of form \upperRomannumeral{2} with
	$$\dfrac{d_H(B{q},B^2{q})}{\mathrm{diam}(B{q})}<\eta.$$
	
	\noindent\textit{Case 2:} $n<0$.
	
	In this case, we have $2k-n>2k$ and
	$$d(g^{2k-n}q,q)=d(g^{2k}q,g^nq)<\eta\cdot\mathrm{diam}(Aq).$$
	Now $A':=\{e,g^{\pm 1},...,g^{\pm (2k-n-1)}\}$ satisfies
	$$\dfrac{d_H(A'{q},(A')^2{q})}{\mathrm{diam}(A'{q})}<\eta$$
	and the condition in Case 1. By the same method as in Case 1, we are able to construct a desired subset $B$.
	
	Conversely, if we have $B=\{\ \exp(tv)\ |\ t\in [-1,1]\}$ with
	$$\dfrac{d_H(B{q},B^2{q})}{\mathrm{diam}(B{q})}<\eta.$$
	For each positive integer $k$, define $B_k=\{\exp(\pm \frac{j}{k}v)\ |\ j=0,\pm 1,...,\pm k \}$. It is clear that $B_k q$ converges to $Bq$ in the Hausdorff sense. Thus for $k$ sufficiently large, $A:=B_k$ is a one-parameter symmetric subset of form \upperRomannumeral{1} with the desired property.
\end{proof}

\subsection{Free action}

We deal with a special case of dimension monotonicity in this section: $G$-action is free at ${p}$.

\begin{prop}\label{dim_free}
	Under the assumptions of Theorem \ref{dimension}, if in addition that $G$ action is free at ${p}$, then\\
	(1) $\dim(G')\le\dim(G)$,\\
	(2) $G'$ has no nontrivial compact subgroups.
\end{prop}

It is direct to prove (2) in Proposition \ref{dim_free}:

\begin{proof}[Proof of Proposition \ref{dim_free}(2)]
	Suppose that $G'$ has a nontrivial compact subgroup $K$. Without lose of generality, we may assume that $K$ is a finite group of prime order $k$. Let $\gamma$ be a generator of $K$. We choose a sequence of elements $\gamma_i\in\Gamma_i$ converging to $\gamma$, and consider the symmetric subset $A_i=\{e,\gamma_i^{\pm 1},...,\gamma_i^{\pm (k-1)}\}$. Clearly,
	$$(r_i{M}_i,{p}_i,A_i)\overset{GH}\longrightarrow({X}',{p}',K).$$
	Before rescaling $r_i$, since $\mathrm{diam}(A_i{p}_i)\to 0$ and $G$-action is free at ${p}$, we conclude that $A_i\to \{e\}$. By Proposition \ref{non_van_prop}(1), $\gamma$ cannot fix ${p}'$. With respect to the metric $r_i{M}_i$, we see
	$$\mathrm{diam}(A_i{p}_i)\to\mathrm{diam}({K{p}'})\ge d(p',\gamma p')>0.$$ Also $d_H(A_i{p}_i,A_i^2{p}_i)\to 0$ because $A_i$ converges to a subgroup $K$. This gives
	$$\dfrac{d_H(A_i{p}_i,A_i^2{p}_i)}{\mathrm{diam}(A_i{p}_i)}\to 0.$$ However, $D_{1,{p}_i}(A_i)<\epsilon$ for $i$ large. A contradiction to Proposition \ref{non_van_prop}(3).
\end{proof}

\begin{cor}\label{dim_free_no_isotropy}
	Under the assumptions of Proposition \ref{dim_free}, $G'$-action is free.
\end{cor}

\begin{proof}
	Otherwise $G'$ would have a nontrivial isotropy subgroup, which is compact.
\end{proof}

\begin{lem}\label{dim_free_untwisted}
	Under the assumptions of Proposition \ref{dim_free}, $G'$-action has no $\eta$-subgroup of one-parameter at $p'$, where $\eta$ is the constant in Proposition \ref{non_van_prop}(3).
\end{lem}

\begin{proof}
	Suppose that $G'$ has an $\eta$-subgroup of one-parameter at $p'$, that is, a symmetric subset $A$ of $G'$ with $\mathrm{diam}(Ap)\in(0,\infty)$ and $$\dfrac{d_H(Ap',A^2p')}{\mathrm{diam}(Ap')}<\eta.$$
	Pick a sequence of symmetric subsets $A_i\subseteq\Gamma_i$ such that
	$$(r_i{M}_i,p_i,A_i)\overset{GH}\longrightarrow({X}',p',A).$$
	By a similar argument we used in Proposition \ref{dim_free}(2), before rescaling $r_i$, we have $D_{1,p_i}(A_i)\to 0$ but $\dfrac{d_H(A_i p_i,A_i^2 p_i)}{\mathrm{diam}(A_i p_i)}<\eta$ for $i$ large. A contradiction.
\end{proof}

\begin{lem}\label{dim_free_gap_lem}
	Let $(Y,q,G)$ be a space and $g$ be an element in $G$. Suppose that $\langle g \rangle$-action is free at $q$ and has no $\eta$-subgroup of one-parameter at $q$. If $d(q,gq)\ge r$ and $d(q,g^Nq)\le R$ for some $N$, then\\
	(1) $d(q,g^jq)\ge\eta r$ for all j. In particular, $\langle g\rangle q$ is $\eta r$-disjoint;\\
	(2) $d(q,g^j q)\le \eta^{-1}R$ for all $-N<j<N$;\\
	(3) there is a constant $C=C(n,\eta,r,R)$ such that $N\le C$.
\end{lem}

\begin{proof}
	(1) If $d(q,g^jq)<\eta r$ for some $j$, we consider $A=\{e,g^{\pm 1},...,g^{\pm j}\}$. Then $\mathrm{diam}(Aq)\ge d(q,gq)\ge r$. Thus
	$$\dfrac{d_H(Aq,A^2q)}{\mathrm{diam}(Aq)}<\dfrac{\eta r}{r}=\eta.$$ A contradiction.
	
	(2) This time we put $A=\{e,g^{\pm 1},...,g^{\pm N}\}$. Then
	$$\mathrm{diam}(Aq)\le \eta^{-1} d_H(Aq,A^2q)\le\eta^{-1} d(q,g^N q)=\eta^{-1}R.$$
	
	(3) This follows from (1),(2), relative volume comparison (of a renormalized limit measure), and a standard packing argument.
\end{proof}

\begin{rem}\label{dim_free_rem_1}
	To prove Lemma \ref{dim_free_gap_lem}(3) only, the assumptions in Lemma \ref{dim_free_gap_lem} can be weakened. Instead of assuming that $\langle g\rangle$-action has no $\eta$-subgroup of one-parameter at $q$, we can assume the following condition:
	
	\textit{For every nontrivial symmetric subset $B$ of $A=\{e,g^{\pm 1},...,g^{\pm N}\}$, we have
	$$\dfrac{d_H(Bq,B^2q)}{\mathrm{diam}(Bq)}\ge\eta.$$} Under this condition, we can show that the points $\{q,g^{ 1}q,...,g^{N}q\}$ are $\eta r$-disjoint by the similar method. The remaining proof is the same.
\end{rem}

\begin{rem}\label{dim_free_rem_1'}
	If $Y\in \mathcal{M}(n,-1)$ is a limit space of a sequence of manifolds $M_i$ with $\mathrm{Ric}_{M_i}\ge-(n-1)\epsilon_i\to 0$, then the constant $C$ in Lemma \ref{dim_free_gap_lem} only depends on $n$, $\eta$, and $R/r$. This follows from the relative volume comparison when Ricci lower bound goes to zero.
\end{rem}

We prove a key lemma for Proposition \ref{dim_free}(1), which states that there exists an equivariant Gromov-Hausdorff distance gap between any $\mathbb{R}^k$-actions with no almost subgroups and any $(\mathbb{R}^k\times\mathbb{Z})$-actions.

\begin{lem}\label{dim_free_gap}
	There exists a constant $\delta(n,\eta)>0$ such that the following holds.
	
	Let $(Y,q,G)$ be a space such that $G=\mathbb{R}^k$ and $G$-action has no $\eta$-subgroup of one-parameter at $q$. Let $(Y',q',G')$ be another space with\\
	\textit{(C1)} $G'$ contains $\mathbb{R}^k\times\mathbb{Z}$ as a closed subgroup,\\
	\textit{(C2)} the extra $\mathbb{Z}$ subgroup has generator whose displacement at $q'$ is less than $1$.
	
	Then $$d_{GH}((Y,q,G),(Y',q',G'))>\delta(n,\eta).$$
\end{lem}

\begin{proof}
	Recall that we always assume that $(Y,q)\in\mathcal{M}(n,-1)$, so it is clear that $k\le n$. We first select a basis of $\mathbb{R}^k$ as follows. Fix any element $v_1\not=e$ in $\mathbb{R}^k$. There is $t_1>0$ such that $d(t_1v_1q,q)=1/n$ and $d(tv_1q,q)<1/n$ for all $t\in (0,t_1)$. Put $e_1=t_1v_1$ as the first element in the basis. Consider the quotient space $(Y/\mathbb{R}e_1, \bar{q}, \mathbb{R}^{k-1})$. Select an element $\bar{e}_2\in\mathbb{R}^{k-1}$ such that $d(\bar{e}_2\bar{q},\bar{q})=1/n$ and $d(t\bar{e}_2\bar{q},\bar{q})<1/n$ for all $t\in(0,1)$. $\bar{e}_2$ corresponds to a coset in $\mathbb{R}^k$. In this coset, choose $e_2$ such that $d(e_2q,q)=d(\bar{e}_2\bar{q},\bar{q})$. By our choice of $e_2$, it is easy to see that $d(te_2q,q)=d(t\bar{e}_2\bar{q},\bar{q})$ for all $t\in(0,1)$. Continue this process until we obtain a basis $\{e_1,...,e_k\}$ in $\mathbb{R}^k$.
	
	We claim that the basis we choose has the following property: for $z=\sum_{j=1}^k\alpha_j e_j$ with $|\alpha_j|\le 1$ for all $j$ and $|\alpha_m|=1$ for some $m$, we have $d(zq,q)\ge r(n,\eta)$, where $r(n,\eta)>0$ is a small constant. In fact, first notice that by our choice of $e_m$, $d((\sum_{j=1}^m\alpha_je_j)q,q)\ge d(e_j q,q)=1/n.$ If $d(\alpha_{m+1} e_{m+1}q,q)<1/2n$, then clearly $d((\sum_{j=1}^{m+1}\alpha_je_j)q,q)\ge 1/2n$. If $d(\alpha_{m+1} e_{m+1}q,q)\ge 1/2n$, by Lemma \ref{dim_free_gap_lem},
	$$|\alpha_{m+1}|\ge \frac{1}{2C(n,\eta,1/2n,1/n)}=:r_1(n,\eta).$$
	Consequently, $d((\sum_{j=1}^{m+1}\alpha_je_j)q,q)\ge r_1(n,\eta).$ Iterate this process at most $k-m-1(<n)$ times, we result in the desired estimate $d(zq,q)\ge r(n,\eta)$.
	
	We set $\delta=1/100$ now and will further modify it later. Let $L=\langle e_1,..,e_k\rangle$ be the lattice generated by $e_1,...,e_k$. Notice that $Lq$ is $1$-dense in the orbit $Gq$. Let $e'_j\in G'$ be an element $\delta$-close to $e_j$ $(j=1,...,k)$. Let $L':=\langle e'_1,...,e'_k\rangle$ be the subgroup of $G'$ generated by these elements. Notice that conditions \textit{(C1)(C2)} guarantee that there is
	$w'\in G'$ such that $d(w'q',q')=d(w'q',L'q')\in (8,10)$. Let $w\in G=\mathbb{R}^k$ be the element $\delta$-close to $w'$. Since $Lq$ is $1$-dense in $Gq$, there is $v\in L$ such that $d(v,w)<1$. We write $v=\sum_{j=1}^k \beta_je_j$ $(\beta_j\in\mathbb{Z})$. Put $M:=\max_j(|\beta_j|)$ and $z=\frac{1}{M}v$. Then $z=\sum_{j=1}^k\alpha_je_j$ with $|\alpha_j|\le 1$ for all $j$ and $|\alpha_m|=1$ for some $m$. By our choice of $\{e_1,...,e_k\}$, we have $d(zq,q)\ge r(n,\eta)$. Also, $d(Mzq,z)\le 12$. Apply Lemma \ref{dim_free_gap_lem}, we conclude that $M\le C_0(n,\eta)$. Consequently, if we set $\delta$ with $nC_0(n,\eta)\delta\le 1/100$, then $v':=\sum_{j=1}^k \beta_je'_j$ is $1/100$-close to $v$. This leads to a contradiction because $d(v'q',L'q')>6$.
\end{proof}

\begin{rem}\label{dim_free_rem_2}
	Inspecting the proof above, we see that only property (3) in Lemma \ref{dim_free_gap_lem} is applied. Hence we may replace the condition that $\mathbb{R}^k$-action has no $\eta$-subgroup of one-parameter at $q$ by the following one:
	
	\textit{There exists a function $C(r,R)>0$ such that for all $z\in \mathbb{R}^k$ with $d(zq,q)\ge r$ and $d(Nzq,q)\le R$, we have $N\le C(r,R)$. }
	
	Correspondingly, the equivariant Gromov-Hausdorff distance gap $\delta$ will depend on $n$ and the function $C$.
\end{rem}

\begin{lem}\label{dim_free_tangent}
	Under the assumption of Proposition \ref{dim_free}, for any $s_j\to\infty$, passing to a subsequence if necessary we consider a tangent cone at ${p}$:
	$$(s_j{X},{p},G)\overset{GH}\longrightarrow(C_{p}{X},v,G_{p}).$$
	Then $G_{p}=\mathbb{R}^{\dim(G)}$.
\end{lem}

\begin{proof}
	We prove the case $G=\mathbb{R}^k$. For the general case, we consider pseudo-action instead and the proof is similar. We know that $G_{p}$ has no nontrivial compact subgroups from Proposition \ref{dim_free}(2). It is also clear that $G_{p}$ contains $\mathbb{R}^k$. As a result, if $G_{p}$ is not $\mathbb{R}^k$, it must contain $\mathbb{R}^k\times \mathbb{Z}$ as a closed subgroup. To prove that $G_{p}=\mathbb{R}^k$, it is enough to show the following: There is $\delta_0>0$, which depends on $({X},{p},G)$, such that for any $s\ge 1$ and for any space $(Y',q',G')$ with\\
	\textit{(C1)} $G'$ contains $\mathbb{R}^k\times\mathbb{Z}$ as a closed subgroup,\\
	\textit{(C2)} the extra $\mathbb{Z}$ subgroup has generator whose displacement at $q'$ is less than $1$,\\
	then
	$$d_{GH}((s{X},{p},G),(Y',q',G'))\ge \delta_0.$$
	By Remark \ref{dim_free_rem_2}, it suffices to prove the following claim.
	
	\noindent\textbf{Claim:} There exists a positive function $C(r,R)$ such that for any $\tau\in(0,1]$ and any $z\in\mathbb{R}^k$ with $d(z{p},{p})\ge \tau r$ and $d(Nz{p},{p})\le \tau R$, we have $N\le C(r,R)$.
	
	For $r>0$, we define
	$$A(r)=\{v\in\mathbb{R}^k \ |\ d(v{p},{p})=r,\  d(tv{p},{p})\le r \text{ for all } 0<t<1 \},$$
	It is clear that $A(r)$ is compact. For $R\ge r$, we define a function on $A(r)$:
	\begin{align*}
		F_{r,R}:A(r)&\to\ \ \mathbb{R}^+\\
		v\ &\mapsto \sup\{t>0\ |\ d(tv{p},{p})= R\}.
	\end{align*}
	Since $\mathbb{R}^k$ is a closed subgroup, $F_{r,R}(v)$ exists and is finite for each $v\in A(r)$. Though $F_{r,R}$ may not be continuous in general, we can check that it is always upper semi-continuous. In fact, given $v_j\in A(r)$ with $v_j\to v$, we put $t_j=F_{r,R}(v_j)$ for simplicity. Then
		$d(t_jv_j{p},{p})=R$ and $d(tv_j{p},{p})>R$ for all $t>t_j$.
	It is clear that $\limsup\limits_{j\to\infty} t_j<\infty$. Since $d(tv{p},{p})\ge R$ for all $t>\limsup\limits_{j\to\infty} t_j$, we conclude that $\limsup\limits_{j\to\infty} t_j\le F_{r,R}(v)$. Let $M_{r,R}<\infty$ be the maximum of $F_{r,R}$ on $A(r)$. If we have $z\in\mathbb{R}^k$ with $d(z{p},{p})\ge r$ and $d(Nz{p},{p})\le R$, then $N\le M_{r,R}$. Let $\tau_0>0$ be a very small number that will be determined later. By our construction of $F_{r,R}$, we see that
	$$M_{\tau r,\tau R}\le M_{\tau_0 r,R}$$
	for all $\tau\in[\tau_0,1]$. This shows that claim holds for $\tau\in[\tau_0,1]$ with positive function $C(r,R)=M_{\tau_0r,R}$. It remains to prove that claim also holds when $\tau\in(0,\tau_0]$ for sufficiently small $\tau_0$.
	
	For $\rho>0$, we further define
	\begin{align*}
		\Omega(\rho)=\{tv\ |\ & t\in [0,1],\ v\in\mathbb{R}^k \text{ with } d(v{p},{p})=\rho\\
		&\text{and } d(sv{p},{p})>\rho \text{ for all } s>1 \}.
	\end{align*}
	Observe that $D_1(\Omega(\rho))\to 0$ as $\rho\to 0$. Thus there is $\tau_0>0$ small such that $$D_1(\Omega(\tau))<\epsilon$$ 
	for all $\tau\le\tau_0$, where $\epsilon=\epsilon(n,v,\Phi)>0$ is the constant in Proposition \ref{non_van_prop}(3). By Proposition \ref{non_van_prop}(3), for any symmetric subset $B\not=\{e\}$ of $\Omega(\tau_0)$, we have
	$$\dfrac{d_H(B{p},B^2{p})}{\mathrm{diam}(B{p})}\ge\eta.$$
	By Remarks \ref{dim_free_rem_1} and \ref{dim_free_rem_1'}, there is some constant $C_0(n,\eta,R/r)$ such that the claim holds for $\tau\in(0,\tau_0]$. Put $C(r,R)=\max\{C_0(n,\eta,R/r),M_{\tau_0 r,R}\}$ and we finish the proof of the claim.
\end{proof}

\begin{rem}
    In Lemma \ref{dim_free_tangent}, we have
    $$(s_j{X},{p},G)\overset{GH}\longrightarrow(C_{p}{X},v,G_{p})$$
    with $G_{p}=\mathbb{R}^{\dim(G)}$. By Lemma \ref{dim_free_untwisted}, we conclude that $G_p$-action has no $\eta$-subgroups at $p$. Recall that when $(X,p)\in\mathcal{M}(n,-1,v)$, $C_pX$ is a metric cone (Theorem \ref{CC_cone}). If one further take this into account, it can be shown that $G_p$ acts as translations in the Euclidean factor of $C_pX$. Since we never used the metric cone structure or any other non-collapsing results in the proof of Lemma \ref{dim_free_tangent}, this lemma can also be applied to the collapsed limit spaces (cf. Theorem \ref{dimension'}). 
\end{rem}

Now we prove Proposition \ref{dim_free}(1) by induction on $\dim(G)$.

\begin{proof}[Proof of Proposition \ref{dim_free}(1)]
	We first show that statement holds when $\dim(G)=0$. In this case, we claim that $G'=\{e\}$. In fact, suppose that $G'$ has an nontrivial element $g'$, then we pick $\gamma_i\in\Gamma_i$ converging to $g'$. Because $G$-action is free at ${p}$, before rescaling $\gamma_i\to e\in G$. By the proof of Corollary \ref{stable_nss}, $\gamma_i=e$ for $i$ large. Hence $\gamma_i$ cannot converge to $g'\not= e$ after rescaling.
	
	Assuming that the statement also holds for $\dim(G)=1,...,k-1$, we verify the case $\dim(G)=k$.
	
	We make the following reductions: by Lemma \ref{dim_free_tangent} and a standard diagonal argument, we assume that $$(t_i{M}_i,{p}_i,\Gamma_i)\overset{GH}\longrightarrow(C_{p}{X},v,G_{p})$$ for some $t_i\to\infty$ with $r_i/t_i\to \infty$, where $C_p{X}$ is a tangent cone at ${p}$ and $G_p=\mathbb{R}^k$. By Lemma \ref{dim_free_untwisted}, $G_{p}$-action has no $\eta$-subgroup of one-parameter at $v$. Now we replace $$(M_i,{p}_i,\Gamma_i)\overset{GH}\longrightarrow({X},{p},G)$$ by $$(t_i{M}_i,{p}_i,\Gamma_i)\overset{GH}\longrightarrow(C_{p}{X},v,\mathbb{R}^{k})$$
	and continue the proof.
	
	We know that $G'_0=\mathbb{R}^{l}$ because it is abelian and has no nontrivial compact subgroup. We show that $l\le k$. Suppose that the contrary holds, that is, $G'$ contains $\mathbb{R}^{k+1}$ as a closed subgroup. Then $G'$ would contain $\mathbb{R}^k\times\mathbb{Z}$ as a closed subgroup. Scaling the sequence $r_i$ down by a constant, we may assume that for the extra $\mathbb{Z}$ subgroup, its generator has displacement at ${p}'$ less than $1$.
	
	Put $\delta(n,\eta)>0$ as the constant in Lemma \ref{dim_free_gap}. For each $i$, consider the following set of scales
	\begin{align*}
		S_i:=\{\  1\le s\le r_i\ |\ & d_{GH}((s{M}_i,{p}_i,\Gamma_i),(Y,q,H))\le \delta/3 \text{ for some space } (Y,q,H)\\
		&  \text{ satisfying the following conditions}\\
		& \textit{(C1)}\ H \text{ contains $\mathbb{R}^k\times\mathbb{Z}$ as a closed subgroup},\\
		& \textit{(C2)}\ \text{this extra $\mathbb{Z}$ subgroup of $H$ has generator whose }\\
		&\ \ \ \ \ \text{\ \ displacement at $q$ is less than $1$.}\}
	\end{align*}
	$S_i$ is nonempty for $i$ sufficiently large because $r_i\in S_i$. We pick $s_i\in S_i$ with $$\inf{S_i}\le s_i\le\inf{S_i}+1/i.$$
	
	\textbf{Step 1:} $s_i\to\infty$.
	
	Otherwise, passing to a subsequence if necessary, $s_i\to s<\infty$. Then
	$$(s_i{M}_i,{p}_i,\Gamma_i)\overset{GH}\longrightarrow(s{X},{p},\mathbb{R}^k).$$
	Since $s_i\in S_i$, for each $i$, there is $(Y_i,q_i,H_i)$ with \textit{(C1)(C2)} and
	$$d_{GH}((s_i{M}_i,{p}_i,\Gamma_i),(Y_i,q_i,H_i))\le \delta/3.$$
	Hence for $i$ large,
	$$d_{GH}((Y_i,q_i,H_i),(s{X},{p},\mathbb{R}^k))\le \delta/2.$$
	This would contradict Lemma \ref{dim_free_gap} because $\mathbb{R}^k$-action on $sX$ has no $\eta$-subgroup of one-parameter at $p$.
	
	\textbf{Step 2:} $r_i/s_i\to\infty$.
	
	If $r_i/s_i\le C$ for some $C\ge 1$. Then consider $$(\dfrac{r_i}{2C}{M}_i,{p}_i,\Gamma_i)\overset{GH}\longrightarrow(\dfrac{1}{2C}{X}',{p}',G').$$
	Note that this limit space satisfies \textit{(C1)(C2)} as well. Thus $r_i/2C\in S_i$ for $i$ large, which contradicts $r_i/\inf(S_i)\le C$.
	
	Next we consider the convergence
	$$(s_i{M}_i,{p}_i,\Gamma_i)\overset{GH}\longrightarrow(Y_\infty,q_\infty,H_\infty)$$
	after passing to a subsequence if necessary.
	
	\textbf{Step 3:} $H_\infty$ contains $\mathbb{R}^{k}$ as a proper closed subgroup.
	
	By Proposition \ref{dim_free}(2), $H_\infty$ does not contain any nontrivial compact subgroups and thus $(H_\infty)_0=\mathbb{R}^{m}$. If $m<k$, we consider
	$$(s_i{M}_i,{p}_i,\Gamma_i)\overset{GH}\longrightarrow(Y_\infty,q_\infty,H_\infty)$$
	and its rescaling sequence ($r_i/s_i\to\infty$)
	$$(r_i{M}_i,{p}_i,\Gamma_i)\overset{GH}\longrightarrow({X}',{p}',G')$$
	with $G'$ containing $\mathbb{R}^{k+1}$ ($k>m$). This contradicts the induction assumptions. It remains to rule out the case $H_\infty=\mathbb{R}^k$ to finish Step 3. By Lemma \ref{dim_free_untwisted}, $H_\infty$-action has no $\eta$-subgroup of one-parameter at $q_\infty$. Together with the fact that $s_i\in S_i$, \textit{(C2)} and Lemma \ref{dim_free_gap}, we can rule out this case.
	
	\textbf{Step 4:}
	We claim that $H_\infty$ contains $\mathbb{R}^k\times\mathbb{Z}$ as a closed subgroup. If this claim holds, we draw a contradiction as follows. Let $h$ be the generator of this extra $\mathbb{Z}$ subgroup. Put $l=d(hq_\infty,q_\infty)>0$. If $l\le 1$, then we choose $t_i=s_i/2\to\infty$. Then $$(t_i{M}_i,{p}_i,\Gamma_i)\overset{GH}\longrightarrow(\dfrac{1}{2}Y_\infty,q_\infty,H_\infty).$$ Hence $t_i\in S_i$ for $i$ sufficiently large. But $t_i<\inf(S_i)$, which is a contradiction. If $l>1$, then we put $t_i=s_i/2l$ and we will result in a similar contradiction.
	
	It remains to verify the claim that $H_\infty$ contains $\mathbb{R}^k\times\mathbb{Z}$ as a proper closed subgroup. From Step 3, we know that $H_\infty$ contains $\mathbb{R}^k$. If $\dim(H_\infty)>k$, since $H_\infty$ is abelian and has no nontrivial compact subgroups, then $H_\infty$ contains $\mathbb{R}^{k+1}$ and the claim follows. If $\dim(H_\infty)=k$, then $H_\infty$ contains $\mathbb{R}^k\times\mathbb{Z}$ by Proposition \ref{dim_free}(2).
\end{proof}

In the proof above, we start with $G'$ containing $\mathbb{R}^{k+1}$ as a closed subgroup and then choose a closed $\mathbb{R}^k\times \mathbb{Z}$ subgroup of $G'$. Through the proof, this closed $\mathbb{R}^k\times \mathbb{Z}$ subgroup ends in a contradiction. This gives the following proposition.

\begin{prop}\label{dim_free_Z}
	Under the assumptions of Proposition \ref{dim_free}, if in addition $\dim(G')=\dim(G)$, then $G'$ is connected.
\end{prop} 

\begin{proof}
	Let $k=\dim(G)=\dim(G')$. Suppose that $G'$ is not connected. Since $G'$ is abelian and does not have any nontrivial compact subgroups, $G'$ must contain $\mathbb{R}^k\times \mathbb{Z}$ as a closed subgroup. This cannot happen as we have seen in the proof of Proposition \ref{dim_free}(1).  
\end{proof}

\begin{rem}
    The proof of Proposition \ref{dim_free}(1) is a prototype for the proof of the general case. Here we choose a critical rescaling sequence with limit $(Y_\infty,q_\infty,H_\infty)$, then make use of Proposition \ref{dim_free}(2), Lemma \ref{dim_free_untwisted}, and Lemma \ref{dim_free_gap} to rule out every possibility of $(Y_\infty,q_\infty,H_\infty)$. When dealing with general $G$-action, we will first extend Proposition \ref{dim_free}(2) and Lemma \ref{dim_free_untwisted} (see Proposition \ref{dim_cpt} and Lemma \ref{dim_ind_untwisted}), then apply a similar argument as the proof of Proposition \ref{dim_free}(1). This method of critical rescaling is also used in \cite{Pan17b}.
\end{rem}

\subsection{Compact subgroups of $G'$}

We look into the compact subgroups of $G'$ and prove Theorem \ref{dimension}(2) in this section. By Proposition \ref{dim_free}(2), we know that if $G'$ has nontrivial compact subgroups, then $G$-action must have nontrivial isotropy subgroups at $p$. We restate Theorem \ref{dimension}(2) here for convenience:

\begin{prop}\label{dim_cpt}
	Suppose that $G'$ has a compact subgroup $K'$. Then $G$ contains a subgroup $K$ fixing ${p}$ and $K$ is isomorphic to $K'$.
\end{prop}

Since $K$ is abelian and compact, it is enough to show that $K/K_0$ is isomorphic to $K'/K'_0$ and $K_0=K'_0=\mathbb{T}^l$ for some $l$.

\begin{rem}
	In fact, $K_0\simeq K'_0$ and $\#K/K_0\ge\#K'/K'_0$ are sufficient for applications.
\end{rem}

\begin{lem}\label{rescal_id_cpt}
	Suppose that  $f_i\in\Gamma_i$ and $({M}_i,{p}_i,f_i)\overset{GH}\longrightarrow({X},{p},\mathrm{id})$ and. Let $r_i\to\infty$ be a rescaling sequence. After passing to a subsequence, we have $(r_i{M}_i,{p}_i,f_i)\overset{GH}\longrightarrow({X}',{p}',f)$. If $\overline{\langle f\rangle}$ is a compact group, then $f=e$.
\end{lem}

\begin{proof}
	Suppose that $f\not=e$. Since $$({M}_i,{p}_i,f_i)\overset{GH}\longrightarrow({X},{p},\mathrm{id}),$$ there is a sequence $k_i\to\infty$ slowly such that $A_i:=\{e,f_i^{\pm 1},...,f_i^{\pm k_i}\}\overset{GH}\to\{e\}$. But after rescaling $r_i$, the limit of $A_i$ contains a compact subgroup $\overline{\langle f\rangle}$. By the same argument in the proof of Proposition \ref{dim_free}(2), we end in a contradiction to Proposition \ref{non_van_prop}(3).
\end{proof}

\begin{lem}\label{dim_circle}
	Let $\mathcal{S}$ be a circle subgroup in $G'_0$, then there is a sequence of symmetric subsets $A_i\subseteq\Gamma_i$ such that $$(r_i{M}_i,{p}_i,A_i)\overset{GH}\longrightarrow({X}',{p}',\mathcal{S})$$
	and before rescaling $$({M}_i,{p}_i,A_i)\overset{GH}\longrightarrow({X},{p},A_\infty)$$ with $A_\infty$ fixing ${p}$ and containing a circle group.
\end{lem}

\begin{proof}
	Select an element $\gamma'\in\mathcal{S}$ with $\overline{\langle\gamma'\rangle}=\mathcal{S}$ and a sequence $\gamma_i\in\Gamma_i$ with $$(r_i{M}_i,{p}_i,\gamma_i)\overset{GH}\longrightarrow({X}',{p}',\gamma').$$ Put $A_i:=\{e,\gamma_i^{\pm 1},...,\gamma_i^{\pm k_i}\}$, where $k_i\to\infty$ slowly such that $$(r_i{M}_i,{p}_i,A_i)\overset{GH}\longrightarrow({X}',{p}',\mathcal{S}).$$
	
	Before rescaling $r_i$, let $A_\infty$ be the limit of $A_i$ and $\gamma$ be the limit of $\gamma_i$. By Lemma \ref{rescal_id_cpt}, $\gamma\not=e$. Moreover, $A_\infty$ fixes ${p}$ because after rescaling $\mathrm{diam}(\mathcal{S}{p}')<\infty$. We claim that $\gamma$ has infinite order. In fact, suppose that $\gamma$ has finite order. Let $N$ be the order of $\langle\gamma\rangle$, then $$({M}_i,{p}_i,\gamma^N_i)\overset{GH}\longrightarrow({X},{p},\mathrm{id}).$$ But after rescaling $r_i$, we have $$(r_i{M}_i,{p}_i,\gamma^N_i)\overset{GH}\longrightarrow({X}',{p}',(\gamma')^N).$$ Since $(\gamma')^N\not=e$, by Lemma \ref{rescal_id_cpt} we result in a contradiction.
	
	Since $\gamma$ has infinite order and $\overline{\langle\gamma\rangle}$ is contained in the isotropy subgroup at ${p}$, we know that $\overline{\langle\gamma\rangle}$ is compact and thus contains a circle $S^1$. It is clear that $A_\infty$ contains $\overline{\langle\gamma\rangle}$. We complete the proof.
\end{proof}

\begin{lem}\label{dim_torus}
	Let $\mathbb{T}^l$ be a torus subgroup of $G'$. Then $G$ also contains $\mathbb{T}^l$, whose action fixes ${p}$.
\end{lem}

\begin{proof}
	Let $\mathcal{S}_j$ $(j=1,...,l)$ be the $j$-th circle factor in $\mathbb{T}^l$. For each $j$, by the proof of lemma \ref{dim_circle}, we can choose symmetric subsets $A_{i,j}\subseteq\Gamma_i$ with the following properties:\\
	(1) $(r_i{M}_i,{p}_i,A_{i,j})\overset{GH}\longrightarrow({X}',{p}',\mathcal{S}_j)$;\\
	(2) $A_{i,j}$ is generated a single element $\gamma_{i,j}$: $A_{i,j}=\{e,\gamma_{i,j}^{\pm 1},...,\gamma_{i,j}^{\pm k_{i,j}}\}$;\\
	(3) $({M}_i,{p}_i,A_{i,j})\overset{GH}\longrightarrow({X},{p},A_{\infty,j})$ with $A_{\infty,j}$ fixing ${p}$ and containing a circle $S^1$.
	
	We claim that the set $\cup_{j=1}^l A_{\infty,j}$ generates a torus of dimension at least $l$. We argue this by induction on $j$. By property (3), the claim holds for $l=1$. Assuming it holds for $l$, we consider the case $l+1$. By induction assumption, $\langle\cup_{j=1}^l A_{\infty,j}\rangle$ contains a torus $T$ of dimension $l$. 
	Suppose that $A_{\infty,l+1}\subseteq T$. Recall that $A_{i,j+1}$ is generated by $\gamma_{i,j+1}$ with property (2) for each $j$. Let $\gamma_{l+1}$ be the limit of $\gamma_{i,l+1}$:
	$$(M_i,p_i,\gamma_{i,l+1})\overset{GH}\longrightarrow(X,p,\gamma_{l+1}).$$
	Since $\gamma_{l+1}\in A_{\infty,l+1}\subseteq T$ and $T$ can be generated by $\cup_{j=1}^l A_{\infty,j}$, there exists a sequence $\beta_i=\prod_{j=1}^l \gamma_{i,j}^{p_{i,j}}$ such that $|p_{i,j}|\le k_{i,j}$ and $$({M}_i,{p}_i,\beta_i)\overset{GH}\longrightarrow({X},{p},\gamma_{l+1}).$$ After rescaling $r_i$, $$(r_i{M}_i,{p}_i,\beta_i)\overset{GH}\longrightarrow({X}',{p}',\beta').$$ By our choice of $\beta_i$, its limit $\beta'\not=e$ is outside $\mathcal{S}_{l+1}$. Now consider the sequence $z_i=\beta_i^{-1}\gamma_{i,l+1}$. Before rescaling $z_i\overset{GH}\to e$, while after rescaling $r_i$, $$z_i\overset{GH}\to z'=(\beta')^{-1}\gamma'_{l+1}\not= e.$$ However, $\overline{\langle z'\rangle}$ is a compact group, which is a contradiction to Lemma \ref{rescal_id_cpt}.
\end{proof}

For finite subgroups of $G'$, there is a similar property.

\begin{lem}\label{dim_finite}
	Let $F'$ be a finite group of $G'$, then $G$ contains a subgroup isomorphic to $F'$, whose action fixes ${p}$.
\end{lem}

\begin{proof}
	Let $g'_1,...,g'_k$ be a set of generators of $F'$. We present $F'$ as
	$$\langle g'_1,...,g'_k|R_1,...,R_l\rangle,$$
	where $R_1,...,R_l=e$ are relations among these generators. For each generator $g'_j$, there is sequence $\gamma_{i,j}\in\Gamma_i$ such that
	$$(r_i{M}_i,{p}_i,\gamma_{i,j})\overset{GH}\longrightarrow({X}',{p}',g'_j).$$
    Before rescaling, passing to a subsequence if necessary, we have $$({M}_i,{p}_i,\gamma_{i,j})\overset{GH}\longrightarrow({X},{p},g_j).$$
    In this way, we obtain $k$ elements $g_1,...,g_k$ in $G$. Let $F$ be the subgroup generated by these $k$ elements. It is clear that $F$-action fixes ${p}$. We show that $F$ is isomorphic to $F'$.

    Let $w$ be a word consisting of $g_1,...,g_k$. Correspondingly, we have words $w_i\in\Gamma_i$ and $w'\in G'$ of the same form. Clearly,
    $$({M}_i,{p}_i,w_i)\overset{GH}\longrightarrow({X},{p},w);$$
    $$(r_i{M}_i,{p}_i,w_i)\overset{GH}\longrightarrow({X},{p},w').$$
    Recall that $w'$ generates a finite group. Thus by Lemma \ref{rescal_id_cpt} and Proposition \ref{non_van_prop}(2), $w=e$ if and only if $w'=e$. This shows that $F$ and $F'$ has the same presentation.
\end{proof}

We prove Proposition \ref{dim_cpt}.

\begin{proof}[Proof of Proposition \ref{dim_cpt}]
  Since $K'$ is compact and abelian, $K'$ admits splitting
  $$K'=\mathbb{T}^l\times F,$$
  where $F=K'/K'_0$ is a finite group. By Lemmas \ref{dim_torus} and \ref{dim_finite}, $G$ contains $\mathbb{T}^l$ and $F$, whose actions fixes $p$. Also by the same argument in the proof of Lemma \ref{dim_torus}, it is clear that $F\cap \mathbb{T}^l=\{e\}$ in $G$. Thus $G$ contains a compact subgroup $\mathbb{T}^l\times F$ fixing $p$.
\end{proof}

We finish this section by results on passing isotropy group to any tangent cone. For a $G$-action on a space $(X,p)$, we denote $\mathrm{Iso}({p},G)$ as the isotropy subgroup of $G$ at $p$.

\begin{lem}\label{dim_fix_tangent}
	For $({M}_i,{p}_i,\Gamma_i)\overset{GH}\longrightarrow({X},{p},G)$ and $s_j\to\infty$, passing to a subsequence if necessary we consider a tangent cone at ${p}$:
	$$(s_j{X},{p},G)\overset{GH}\longrightarrow(C_p{X},v,G_{p}).$$
	If $G$ is a compact group fixing ${p}$ with $G_0=\mathbb{T}^l$, then $(G_{p})_0=\mathbb{T}^l$ and $$\#\pi_0(G_{p})\le \#\pi_0(G),$$
	where $\#\pi_0$ means the number of connected components.
\end{lem}

\begin{proof}
	It is clear that $G_{p}$ fixes $v$. We first prove the case $G=\mathbb{T}^l$. By Proposition \ref{dim_cpt}, we know that $G$ contains a subgroup isomorphic to $G_{p}$. Since $G=\mathbb{T}^l$, $G_{p}$ must contain a subgroup of $\mathbb{T}^l$. We show that $(G_{p})_0=\mathbb{T}^l$, which implies that $G_{p}=\mathbb{T}^l$ with the help of Proposition \ref{dim_cpt}. Suppose that $(G_{p})_0=\mathbb{T}^m$ with $m<l$. Notice that $G=\mathbb{T}^l$ contains exactly $2^l-1$ many non-identity elements of order $2$. From the sequence $\{(s_j{X},{p},G)\}_j$, we obtain $2^l-1$ different sequences of elements with order $2$ in $G$. It is clear that, passing to a subsequence if necessary, their limits are contained in $(G_{p})_0$ and have order $2$. On the other hand, $(G_{p})_0=\mathbb{T}^m$ has $2^m-1$ many non-identity elements of order $2$. Thus there must be two sequences $\{\alpha_{1,j}\}$, $\{\alpha_{2,j}\}$ such that
	$$\alpha_{k,j}\not=e,\ \alpha_{k,j}^2=e \  (k=1,2) ,\ \alpha_{1,j}\not=\alpha_{2,j}$$
	but their limits are the same. Then $\beta_j=\alpha_{1,j}\alpha_{2,j}\not=e$ would converge to $e$. On the other hand, $\beta_j$ has order $2$; thus by Theorem \ref{main_no_small_subgroup}, $D_{1,p}(\beta_j)\ge \delta(n,v)>0$ on $sX$ for all $s\ge 1$, a contradiction.
	
	For the general case, $G$ may have multiple components, that is, $G=\mathbb{T}^l\times F$, where $F$ is a finite group. Apply the same argument above, we see that $(G_{p})_0=\mathbb{T}^l$. Now the result follows from Proposition \ref{dim_cpt}. 	
\end{proof}

\begin{rem}
	In Lemma \ref{dim_fix_tangent}, in fact one can show that $G_p$ is isomorphic to $G$. The current statement is sufficient for our purposes.
\end{rem}

\begin{cor}\label{dim_tangent}
	For $({M}_i,{p}_i,\Gamma_i)\overset{GH}\longrightarrow({X},{p},G)$ with $G_0=\mathbb{R}^k\times\mathbb{T}^l$, and $s_j\to\infty$, passing to a subsequence if necessary, we consider a tangent cone at ${p}$:
	$$(s_j{X},{p},G)\overset{GH}\longrightarrow(C_{p}X,v,G_p).$$
	Then $G_{p}=\mathbb{R}^k\times K$, where $K$-action fixes $v$, $K_0=\mathbb{T}^l$ and $$\#\pi_0(K)\le\#\pi_0(\mathrm{Iso}({p},G)).$$
\end{cor}

\begin{proof}
	We put $K$ as the limit of $\mathrm{Iso}({p},G)$ with respect to the sequence
	$$(s_j{X},{p},G)\overset{GH}\longrightarrow(C_{p}{X},v,G_{p}).$$
	With Lemmas \ref{dim_free_tangent} and \ref{dim_fix_tangent}, it remains to check that $G_{p}$ has the splitting $\mathbb{R}^k\times K$. In fact, note that $K\cap\mathbb{R}^k=e$ and $\mathbb{R}^k\cdot K=G_{p}$. Hence the splitting follows.
\end{proof}

\begin{rem}\label{split}
	For a space $(Y,q,H)$, because $H$ is abelian, as long as the orbit $H\cdot q$ is homeomorphic to $\mathbb{R}^k$, we always have the splitting $H=\mathbb{R}^k\times\mathrm{Iso}(q,H)$.
\end{rem}

\subsection{General $G$-action and a triple induction}

We complete the proof of Theorem \ref{dimension} in this section. We make some reductions at first. By Lemma \ref{dim_tangent}, a standard rescaling and diagonal argument, we may pass to a tangent cone of ${X}$ at ${p}$ and assume that $G=\mathbb{R}^k\times\mathrm{Iso}({p},G)$. We will always assume this reduction when proving Theorem \ref{dimension}(1).

For a space $(X,p,G)$ with $G=\mathbb{R}^k\times \mathrm{Iso}(p,G)$, we define $\dim_R(G)=k$ and $\dim_T(G)=\dim(\mathrm{Iso}(p,G))$ as the dimension of $\mathbb{R}$-factors and torus factors in $G$ respectively. We will prove Theorem \ref{dimension} by a triple induction argument on $\dim_T(G)$, $\dim_R(G)$ and $\#\pi_0(G)$. Due to the reduction we made, $\#\pi_0(G)$ equals to the number of connected components of $\mathrm{Iso}(p,G)$. Also note that the case $\dim_T(G)=0$ with $\#G/G_0=1$ is proved as Proposition \ref{dim_free}(1); and the case $\dim_R(G)=\dim_T(G)=0$ follows from Corollary \ref{stable_nss}. When we say such a $G$ in the induction assumptions, we always mean that such a limit group is possible to exist as the limit of $({M}_i^n,{p}_i,\Gamma_i)$ (for example, $\dim_R(G)$ is always no greater than $n$).

When proving each induction, we will also show an extra proposition regarding the extremal case:
\begin{prop}\label{dim_Z}
Under the assumptions of Theorem \ref{dimension}, suppose that\\
(1) $G=\mathbb{R}^k\times K$, where $K=\mathrm{Iso}(p,G)$ (this is the reduction we used);\\
(2) $K'=\mathrm{Iso}(p',G')$ has the same dimension as $K$ and $\#\pi_0(K/K_0)=\#\pi_0(K'/K'_0)$;\\
(3) $G'$ contains $\mathbb{R}^l$ as a closed subgroup.\\
Then $G'=\mathbb{R}^k\times K'$.
\end{prop}

Proposition \ref{dim_Z} generalizes Proposition \ref{dim_free_Z}, which is the case $G=\mathbb{R}^k$ with $G'$ containing $\mathbb{R}^k$. Later, Proposition \ref{dim_Z} will be used together with Theorem \ref{dimension} (as Corollaries \ref{dimension_order} and \ref{dimension_cor}) to bound the number of short generators.

We state the triple induction:

\ 

\noindent\textit{Induction on $\# \pi_0(G)$:} Under the reductions, suppose that Theorem \ref{dimension}(1) and Proposition \ref{dim_Z} hold when\\
(1) $G_0=\mathbb{R}^k\times \mathbb{T}^l$ with $\#\pi_0(G)\le m$, or\\
(2) $\dim_T(G)=l$ with $\dim_R(G)<k$, or\\
(3) $\dim_T(G)<l$.\\
Then it holds for $G_0=\mathbb{R}^k\times \mathbb{T}^l$ with $\#\pi_0(G)=m+1$.

\

\noindent\textit{Induction on $\dim_R(G)$:} Under the reductions, suppose that Theorem \ref{dimension}(1) and Proposition \ref{dim_Z} hold when\\
(1) $\dim_T(G)=l$ with $\dim_R(G)\le k$, or\\
(2) $\dim_T(G)<l$.\\
Then it holds for $G=\mathbb{R}^{k+1}\times \mathbb{T}^l$.

\

\noindent\textit{Induction on $\dim_T(G)$:} Under the reductions, suppose that Theorem \ref{dimension}(1) and Proposition \ref{dim_Z} hold for $\dim_T(G)\le l$, then it holds for $G=\mathbb{T}^{l+1}$.

\

Applying these three inductions above repeatedly, we will eventually cover every possible $G$. More precisely, we start with base case $\dim_R(G)=\dim_T(G)=0$ (see proof of Corollary \ref{stable_nss}). Together with Proposition \ref{dim_free}(1), induction on $\dim_R(G)$ and on $\# G/G_0$, we conclude that Theorem \ref{dimension} holds for any $G=\mathbb{R}^k\times F$, where $F$ is a finite group fixing $p$. Then by induction on $\dim_T(G)$, we know it also holds for $G=S^1$. After that, apply inductions on $\dim_R(G)$ and on $\#\pi_0(G)$ again, and we cover the case $G=\mathbb{R}^k\times \mathrm{Iso}(p,G)$ with $\mathrm{Iso}(p,G)_0=S^1$. We continue this process and finish the proof of Theorem \ref{dimension}(1).

All these three induction arguments are similar to the proof of Proposition \ref{dim_free}(1): choose a critical rescaling sequence and rule out every possibility in the corresponding limit. To illustrate this strategy, we consider the case $G=\mathbb{R}\times S^1$ as an example. By Proposition \ref{dim_torus}, we know that $G'$ has no torus of dimension $>1$. We need to rule out the case like $G'=\mathbb{R}^3$. This $G'$ contains $\mathbb{R}^2\times \mathbb{Z}$ as a closed subgroup. For $\delta>0$ small, we consider
\begin{align*}
	S_i:=\{\  1\le s\le r_i\ |\ & d_{GH}((s{M}_i,{p}_i,\Gamma_i),(Y,q,H))\le \delta \text{ for some space } (Y,q,H)\\
	&  \text{ with $H$-action satisfying the following conditions}\\
	& \textit{(C1)}\ H \text{ contains $\mathbb{R}^2\times\mathbb{Z}$ as a closed subgroup },\\
	& \textit{(C2)}\ \text{This $\mathbb{Z}$ subgroup has generator whose displacement}\\
	&\ \ \ \ \ \text{at $q$ is less than $1$.}
\end{align*}

Pick $s_i\in S_i$ with $\inf(S_i)\le s_i \le \inf{S_i}+1/i$. Assume $s_i\to\infty$ and we consider $$(s_i{M}_i,{p}_i,\Gamma_i)\overset{GH}\longrightarrow(Y_\infty,q_\infty,H_\infty).$$ Like step 4 in the proof of Proposition \ref{dim_free}(1), if $H_\infty$ contains $\mathbb{R}^2\times\mathbb{Z}$ as a closed subgroup, then we will obtain a contradiction by scaling $s_i$ down by a constant. One can also apply induction assumptions to rule out the cases like $H_\infty=\mathbb{R}\times F$ or $H_\infty=S^1$. If $H_\infty=\mathbb{R}\times S^1$ but $S^1$-action is free at $q_\infty$, then we can apply the result in free case. The last case we want to eliminate is that $H_\infty=\mathbb{R}\times S^1$ with $S^1$-action fixing $q_\infty$.

Here comes a distinction between general case and free case in Section 3.1: for general limit $G$-action, rescaling limit group $H_\infty$-action may have $\eta$-subgroups at ${p}'$. The observation is that, if $H_\infty$ contains a torus of the same dimension as $\dim_T(G)$ and this torus fixes $p'$, then actions of $\mathbb{R}^k$ subgroups in $G'$ should have no $\eta$-subgroups of one-parameter at ${p}'$ (see Lemma \ref{dim_ind_untwisted} below for the precise statement). With this in hand, then together with an equivariant $GH$-distance gap between $(Y_\infty,q_\infty,H_\infty)$ and the spaces we used to define $S_i$ (see Lemma \ref{dim_gap}), we can rule out the case $H_\infty=\mathbb{R}\times S^1$ when $\delta$ is sufficiently small.

Following this idea, we prove the lemma below.
\begin{lem}\label{dim_ind_untwisted}
	Suppose that $\mathrm{Iso}(p,G)$ has identity component $\mathbb{T}^l$. Further suppose that $\mathrm{Iso}(p',G')$ contains a torus of dimension $l$, that is, $G'_0=\mathbb{R}^{k}\times \mathbb{T}^l$ with $\mathbb{T}^l$ fixing ${p}'$ (Recall that torus factor in $G'_0$ cannot have dimension $>l$ by Lemma \ref{dim_torus}). Then $\mathbb{R}^{k}$-action on $X'$ has no $\eta$-subgroup of one-parameter at ${p}'$.
\end{lem}

\begin{rem}
	In Lemma \ref{dim_ind_untwisted}, $G'$ contains infinitely many subgroups isomorphic to $\mathbb{R}^k$, but their orbits at $p'$ are exactly the same because $\mathbb{T}^l$ fixes $p'$. Thus the condition that $\mathbb{R}^k$-action has no $\eta$-subgroup of one-parameter at $p'$ has no ambiguity.
\end{rem}

One may regard Lemma \ref{dim_ind_untwisted} as a generalization of Lemma \ref{dim_free_untwisted}, where $\mathrm{Iso}({p},G)$ is trivial (also compare with the proof of Lemma \ref{dim_torus}).

\begin{lem}\label{rescal_id_twisted}
	Let $\eta$ be the constant Proposition \ref{non_van_prop}(3) and let $f_i\in\Gamma_i$. Suppose that the following sequences converge ($r_i\to\infty$) $$({M}_i,{p}_i,f_i)\overset{GH}\longrightarrow({X},{p},\mathrm{id})$$ $$(r_i{M}_i,{p}_i,f_i)\overset{GH}\longrightarrow({X}',{p}',f\not= \mathrm{id}).$$ Then the following can NOT happen: for some integer $k$, $A_\infty=\{e,f^{\pm 1},...,f^{\pm k}\}$ satisfies
	$$\dfrac{d_H(A_\infty p',A_\infty^2 p')}{\mathrm{diam}(A_\infty p')}<\eta.$$
\end{lem}

\begin{proof}
	Suppose that there is $A_\infty=\{e,f^{\pm 1},...,f^{\pm k}\}$ of
	$$\dfrac{d_H(A_\infty p',A_\infty^2 p')}{\mathrm{diam}(A_\infty p')}<\eta.$$
	Put $A_i=\{e,f_i^{\pm 1},...,f_i^{\pm k}\}$, then $$({M}_i,{p}_i,A_i)\overset{GH}\longrightarrow({X},{p},\{e\})$$
	and  $$(r_i{M}_i,{p}_i,A_i)\overset{GH}\longrightarrow({X}',{p}',A_\infty).$$ Clearly this contradicts Proposition \ref{non_van_prop}(3).
\end{proof}

\begin{proof}[Proof of Lemma \ref{dim_ind_untwisted}]
	Suppose that $\mathbb{R}^{k}$-action has an $\eta$-subgroup of one-parameter at ${p}'$. We show that $\mathrm{Iso}(p,G)$ contains $\mathbb{T}^{l+1}$, which contradicts the assumption.
	
	We follow the proof of Lemma \ref{dim_torus}. For each circle factor $\mathcal{S}_j$ in $G'$ $(j=1,...,k)$, we can pick $A_{i,j}=\{e,\gamma_{i,j}^{\pm 1},...,\gamma_{i,j}^{\pm k_{i,j}}\}\subset \Gamma_i$ with properties (1)-(3) as in the proof of Lemma \ref{dim_torus}. We also know that $\{A_{\infty,j}\}_{j=1}^l$ contains $l$ independent circles.
	
	Since $\mathbb{R}^{k}$-action has an $\eta$-subgroup of one-parameter at ${p}'$, it contains some one-parameter symmetric subset $\mathcal{T}$ such that
	$$\dfrac{d_H(\mathcal{T}{p}',\mathcal{T}^2{p}')}{\mathrm{diam}(\mathcal{T}{p}')}<\eta.$$
	By Lemma \ref{1-para_sym_subset_form}, we can assume that $\mathcal{T}$ has form \upperRomannumeral{2}. We write $\mathcal{T}$ as $\{tg\ |\ t\in[-1,1]\}$. Put $F:=\pi_0(\mathrm{Iso}({p},G))$, which is a finite group. We choose a large integer $m_0$ such that $\frac{1}{m_0}g$ satisfies the following property: for any integer $N=1,...,\# F+1$, $\mathcal{T}_{N}:=\{e,\frac{N}{m_0}g,\frac{2N}{m_0}g,...,\frac{k_N N}{m_0}g\}$ satisfies
	$$\dfrac{d_H(\mathcal{T}_{N}{p}',(\mathcal{T}_{N})^2{p}')}{\mathrm{diam}(\mathcal{T}_{N}{p}')}<\eta,$$
	where $k_N$ is the largest integer with $k_N N\le m_0$.
	
	Choose $f_i\in \Gamma_i$ with
	$$(r_i{M}_i,{p}_i,f_i)\overset{GH}\longrightarrow({X}',{p}',\frac{1}{m_0}g).$$
	Let $f$ be a limit of $f_i$ before rescaling. It is clear that $f\in\mathrm{Iso}(p,G)$. By Lemma \ref{rescal_id_twisted}, the know that $f^N\not= e$ for all $N=1,...,\#F+1$.
	
	\noindent\textbf{Claim :} For all $N=1,...,\#F+1$, $f^N$ is outside $\mathbb{T}^l$. 
	
	By the proof of Lemma \ref{dim_torus}, we have $\cup_{j=1}^l A_{\infty,j}$ generates $\mathbb{T}^l\subseteq \mathrm{Iso}(p,G)$. Suppose that $f^N\in\mathbb{T}^l$. Then there is $\beta_i=\prod_{j=1}^l \gamma_{i,j}^{p_{i,j}}$ with $|p_{i,j}|\le k_{i,j}$ such that $$({M}_i,{p}_i,\beta_i)\overset{GH}\longrightarrow({X},{p},f^N).$$ After rescaling $r_i$, $$(r_i{M}_i,{p}_i,\beta_i)\overset{GH}\longrightarrow({X}',{p}',\beta')$$ with $\beta'\in\mathbb{T}^{l}\subseteq \mathrm{Iso}(p',G')$. We consider the sequence $z_i=\beta_i^{-1}f_i^N$. It is clear that $z_i\overset{GH}\to e$, while after rescaling $r_i\to\infty$, $z_i\overset{GH}\to z'\not=e$ because $\beta'\in\mathbb{T}^{l}$ and $\frac{N}{m_0}g$ is in some closed $\mathbb{R}$ subgroup. Put $C=\{e,z^{\pm 1},...,z^{\pm k_N}\}$. Since $\mathbb{T}^{l}$-action fixes ${p}'$, the orbit $C{p}'$ is identically the same as $\mathcal{T}_{N}{p}'$. Apply Lemma \ref{rescal_id_twisted} and we obtain the desired contradiction. This proves the claim.
	
	Since all these $f^N$ $(N=1,...,\#F+1)$ lie in $\mathrm{Iso}({p},G)$, which consists of exactly $\#F$ connected components, there must be some $N$ such that $f^N$ lies inside the identity component $\mathbb{T}^l$, a contradiction to the claim we just showed.
\end{proof}

Besides Lemma \ref{dim_ind_untwisted}, another ingredient to prove the general case is an equivariant Gromov-Hausdorff gap like Lemma \ref{dim_free_gap}. Actually here we only need to modify the statement of Lemma \ref{dim_free_gap}, because we only used the properties of $G$-orbit at $q$ in the proof of Lemma \ref{dim_free_gap}.

\begin{lem}\label{dim_gap}
	There exists a constant $\delta(n,\eta)>0$ such that the following holds.
	
	Let $(Y,q,G)$ be a space with $G=\mathbb{R}^k\times \mathrm{Iso}(q,G)$. Suppose that $\mathbb{R}^k$-action on $Y$ has no $\eta$-subgroup of one-parameter at $q$. Let $(Y',q',G')$ be another space with\\
    \textit{(C1)} $G'$ contains $\mathbb{R}^k\times\mathbb{Z}$ as a closed subgroup,\\
    \textit{(C2)} this extra $\mathbb{Z}$ subgroup has generator whose displacement at $q'$ is less than $1$.
	
	Then $$d_{GH}((Y,q,G),(Y',q',G'))>\delta(n,\eta).$$
\end{lem}

With all these preparations, we start the triple induction described in the beginning of this section. We begin with the easiest one among these three: induction on $\dim_T(G)$. Actually for this one, we do not even need the preparations above.

\begin{proof}[\textit{Proof of Induction on $\dim_T(G)$.}]
	Under the reductions, assuming that the Theorem \ref{dimension}(1) and Proposition \ref{dim_Z} hold when $\dim_T(G)\le l$, we need to verify the case $G=\mathbb{T}^{l+1}$ with $G$ fixing ${p}$. Our goal is the following:\\
	(a) rule out $\dim(G')>l+1$;\\
    (b) if $\mathrm{Iso}(p',G')=\mathbb{T}^{l+1}$, then $G'=\mathbb{T}^{l+1}$.\\ 
	We argue by contradiction, suppose that for some $r_i\to\infty$ and some convergent subsequence
	$$(r_i{M}_i,{p}_i,\Gamma_i)\overset{GH}\longrightarrow({X}',{p}',G'),$$
	we have\\
	(a) $\dim(G')>l+1$, or\\
	(b) $\mathrm{Iso}(p',G')=\mathbb{T}^{l+1}$ is a proper subgroup of $G'$.
	
	For (a), by Lemma \ref{dim_torus}, we know that $G'$ cannot contain a torus of dimension $>l+1$. As a result, if $\dim(G')>l+1$, then $G'$ contains a closed $\mathbb{R}$ subgroup, and thus contains a closed $\mathbb{Z}$ subgroup.
	
	For (b), from Proposition \ref{dim_cpt} we see that any element of $G'$ outside $\mathbb{T}^{l+1}$ has infinite order. We also conclude that $G'$ contains a closed $\mathbb{Z}$ subgroup.
	
	Rescaling $r_i$ down by a constant if necessary, we assume that this $\mathbb{Z}$ subgroup has generator whose displacement at $p'$ is less than $1$. For $\delta=1/10$, we consider the following set of scales for each $i$,
	\begin{align*}
		S_i:=\{\  1\le s\le r_i\ |\ & d_{GH}((s{M}_i,{p}_i,\Gamma_i),(Y,q,H))\le \delta/3 \text{ for some space } (Y,q,H)\\
		&  \text{ satisfying the following conditions}\\
		& \textit{(C1)}\ \text{$H$ contains $\mathbb{Z}$ as a closed subgroup,}\\
		& \textit{(C2)}\ \text{this $\mathbb{Z}$ subgroup has generator whose displacement}\\
		& \ \ \ \ \ \text{\ \ at $q$ is less than $1$.}\}
	\end{align*}
    (see Remark \ref{critical_scales} for explanations on the definition of $S_i$)

    Since $G'$ contains a closed $\mathbb{Z}$ subgroup, we conclude that $r_i\in S_i$ for $i$ large. Pick $s_i\in S_i$ with $\inf(S_i)\le s_i\le \inf(S
	_i)+1/i$.
	
	We show that $s_i\to\infty$. In fact, suppose that $s_i$ subconverges to $s<\infty$, then after passing to a subsequence, we have
	$$(s_i{M}_i,{p}_i,\Gamma_i)\overset{GH}\longrightarrow(s{X},{p},G).$$
	Since $s_i\in S_i$, each $(s_i{M}_i,{p}_i,\Gamma_i)$ is $\delta/3$-close to some space $(Y_i,q_i,H_i)$ with conditions \textit{(C1)(C2)}. $G$ fixes ${p}$ while $H_i$ contains some element $h_i$ moving $q_i$ with displacement less than $1$. Furthermore, by condition \textit{(C1)} the orbit $\langle h_i\rangle q_i$ has infinite diameter. Obviously,  $(Y_i,q_i,H_i)$ cannot be $\delta$ close to $(s{X},{p},G)$. A contradiction.
	
	As Step $2$ in the proof of Proposition \ref{dim_free}, we follow the same argument and conclude that $r_i/s_i\to\infty$.
	
	Now consider the convergent sequence
	$$(s_i{M}_i,{p}_i,\Gamma_i)\overset{GH}\longrightarrow(Y_\infty,q_\infty,H_\infty).$$
	and we make the following observations:\\
	1. If $\mathrm{Iso}(q_\infty,H_\infty)$ has dimension $<l+1$, then we would obtain a contradiction to the induction assumptions by passing to the tangent cone at $q_\infty$ and applying the fact that $r_i/s_i\to\infty$.\\
	2. If $\dim(H_\infty)>l+1$, then $H_\infty$ contains a closed $\mathbb{R}$ subgroup due to Proposition \ref{dim_cpt}. We follow the method used in Step 4 of the proof of Proposition \ref{dim_free} to draw a contradiction. More precisely, we can rescale $s_i$ down by a constant but this smaller rescaling still belongs to $S_i$ for $i$ large, and this leads to a contradiction to our choice of $s_i$.\\
	3. If $H_\infty=\mathbb{T}^l$ fixing $q_\infty$, then we also end in a contradiction. This is because each $(s_i{M}_i,{p}_i,\Gamma_i)$ is $\delta/3$ close to some $(Y_i,q_i,H_i)$, where $H_i$ has some element $h_i$ moving $q_i$ with displacement less than $1$ and $\mathrm{diam}(\langle h_i\rangle q_i)=\infty$. This cannot happen for $\delta=1/10$.
	
	Therefore, the only possible situation left is that, $H_\infty$ contains $(H_\infty)_0=\mathbb{T}^{l+1}$ as a proper subgroup with $\mathbb{T}^{l+1}$-action fixing $q_\infty$. By Proposition \ref{dim_cpt}, $H_\infty$ does not contain any element of finite order outside $(H_\infty)_0$. Thus $H_\infty$ contain a closed $\mathbb{Z}$ subgroup, then we can rule out this case as we did in observation 2 above.
	
	We have ruled out every possibility of $(Y_\infty,q_\infty,H_\infty)$. This completes the proof.
\end{proof}

\begin{rem}\label{critical_scales}
	When defining $S_i$ in the proof above, we only require that $(Y,q,H)$ contains some $\mathbb{Z}$ subgroup moving $q$ (but not too far). So logically, if $G'=\mathbb{R}$, which may happen, then such $S_i$ is still nonempty and we can still pick $s_i$ close to $\inf(S_i)$. However, in this case, we will not find any contradiction. Inspecting the proof above, we used the hypothesis that $G'$ has something extra compared with $G$ to rule out every possibility of $(Y_\infty,q_\infty.H_\infty)$ (for example, in observation 1, we applied the induction assumption).
\end{rem}

Next we prove induction on $\dim_R(G)$.

\begin{proof}[\textit{Proof of Induction on $\dim_R(G)$.}]
	Under the reductions, assuming that Theorem \ref{dimension} and Proposition \ref{dim_Z} hold when\\
	(1) $\dim_T(G)=l$ with $\dim_R(G)\le k$, or\\
	(2) $\dim_T(G)<l$,\\
	we need to show that when $G=\mathbb{R}^{k+1}\times\mathbb{T}^l$ with $\mathbb{T}^l$ fixing $p$, for any rescaling sequence $r_i\to\infty$ and any convergent subsequence
	$$(r_i{M}_i,{p}_i,\Gamma_i)\overset{GH}\longrightarrow({X}',{p}',G'),$$
	we have the following:\\
	(a) $\dim(G')\le (k+1)+l$;\\
	(b) if $G'$ contains $\mathbb{R}^{k+1}\times \mathbb{T}^l$ with $\mathbb{T}^l$ fixing $p'$, then $G'=\mathbb{R}^{k+1}\times \mathbb{T}^l$.
		
	We argue by contradiction. Suppose that there is a rescaling sequence $r_i\to\infty$ such that the corresponding limit group $G'$ has\\
	(a) dimension $>(k+1)+l$, or\\
	(b) $\mathbb{R}^{k+1}\times \mathbb{T}^l$ being a proper subgroup of $G'$, where $\mathbb{T}^l$ fixing $p'$.
	
	For (a), by Proposition \ref{dim_cpt}, we know that $G'$ has no torus factor of dimension $>l$, thus it must contain $\mathbb{R}^{k+2}$ as a closed subgroup. In particular, $G'$ contains a closed subgroup $\mathbb{R}^{k+1}\times\mathbb{Z}$.
	
	For (b), again by Proposition \ref{dim_cpt}, we see that any element of $G'$ outside $\mathbb{R}^{k+1}\times \mathbb{T}^l$ must has infinite order. Thus $G'$ also has a closed subgroup as $\mathbb{R}^{k+1}\times\mathbb{Z}$.
	
	Rescaling $r_i$ down by a constant if necessary, we assume that the extra $\mathbb{Z}$ subgroup has generator whose displacement at $p'$ is less than $1$. Let $\delta=\delta(n,\eta)>0$ be the constant in Lemma \ref{dim_gap}. We consider
	\begin{align*}
		S_i:=\{\  1\le s\le r_i\ |\ & d_{GH}((s{M}_i,{p}_i,\Gamma_i),(Y,q,H))\le \delta/3 \text{ for some space } (Y,q,H)\\
		&  \text{ satisfying the following conditions}\\
		& \textit{(C1)}\ H \text{ contains $\mathbb{R}^{k+1}\times\mathbb{Z}$ as a closed subgroup},\\
		& \textit{(C2)}\ \text{this extra $\mathbb{Z}$ subgroup of $H$ has generator whose }\\
		&\ \ \ \ \text{\ \ \ displacement at $q$ is less than $1$.}\}
	\end{align*}
	We know that $r_i\in S_i$ for $i$ large. Pick $s_i\in S_i$ such that $\inf(S_i)\le s_i\le\inf(S_i)+1/i$.
	
	We show that $s_i\to\infty$. Suppose that $s_i$ sub-converges to $s<\infty$, then
	$$(s_i{M}_i,{p}_i,\Gamma_i)\overset{GH}\longrightarrow(s{X},{p},G).$$
	For $i$ large, since $s_i\in S_i$, there is some space $(Y_i,q_i,H_i)$ with conditions \textit{(C1)(C2)} above and
	$$d_{GH}((s{X},{p},G),(Y_i,q_i,H_i))\le\delta/2.$$
	Recall that by the reductions at the beginning of this section and Lemma \ref{dim_ind_untwisted}, $\mathbb{R}^{k+1}$-action has no $\eta$-subgroup of one-parameter at ${p}$ ($\mathbb{R}^{k+1}\subseteq G$). We apply Lemma \ref{dim_gap} and obtain the desired contradiction.
	
	Following the same proof as Step 2 in Proposition \ref{dim_free}, we derive that $r_i/s_i\to\infty$.
	
	We consider
	$$(s_i{M}_i,{p}_i,\Gamma_i)\overset{GH}\longrightarrow(Y_\infty,q_\infty,H_\infty).$$
	If $\dim(H_\infty)>(k+1)+l$, then $H_\infty$ contains $\mathbb{R}^{k+1}\times\mathbb{Z}$. Following Step 4 in the proof of Proposition \ref{dim_free}, we will get a contradiction by rescaling down $s_i$ by a constant. Thus we must have $\dim(H_\infty)\le (k+1)+l$. If $\dim(H_\infty)<(k+1)+l$, or $\dim(H_\infty)=(k+1)+l$ but $\mathrm{Iso}(q_\infty,H_\infty)$ has dimension $<l$, then we consider
	$$(s_i{M}_i,{p}_i,\Gamma_i)\overset{GH}\longrightarrow(Y_\infty,q_\infty,H_\infty)$$
	and its rescaling sequence ($r_i/s_i\to\infty$)
	$$(r_i{M}_i,{p}_i,\Gamma_i)\overset{GH}\longrightarrow({X}',{p}',G').$$
	Apply the induction assumptions, we rule out such cases. 
	
	The only remaining case is $(H_\infty)_0=\mathbb{R}^{k+1}\times \mathbb{T}^l $ with $\mathbb{T}^l$-action fixing $q_\infty$. By Lemma \ref{dim_ind_untwisted}, $\mathbb{R}^{k+1}$-action has no $\eta$-subgroup of one-parameter at $q_\infty$. If $H_\infty$ is connected, we apply Lemma \ref{dim_gap} once again and end in a contradiction. If $H_\infty$ has finitely many components, then the contradiction arises from Proposition \ref{dim_cpt}. If $H_\infty$ has infinitely many components, then again by Proposition \ref{dim_cpt}, $H_\infty$ contains $\mathbb{R}^{k+1}\times\mathbb{Z}$ as a closed subgroup, which would contradict our choice of $s_i$.
\end{proof}

We finish the proof of Theorem \ref{dimension}(1) by verifying the last induction on $\#\pi_0(G)$.

\begin{proof}[\textit{Proof of Induction on $\#\pi_0(G)$.}]
	Under the reductions, assume that Theorem \ref{dimension}(1) and Proposition \ref{dim_Z} hold when\\
	(1) $G_0=\mathbb{R}^k\times \mathbb{T}^l$ with $\#G/G_0\le m$, or\\
	(2) $\dim_T(G)=l$ with $\dim_R(G)<k$, or\\
	(3) $\dim_T(G)<l$.\\
	We need to verify the case $G_0=\mathbb{R}^k\times \mathbb{T}^l$ with $\#\pi_0(G)=m+1$. By reductions, we assume that $G=\mathbb{R}^k\times\mathrm{Iso}({p},G)$.
	
	We argue by contradiction. Suppose that for some $r_i\to\infty$,
	$$(r_i{M}_i,{p}_i,\Gamma_i)\overset{GH}\longrightarrow({X}',{p},G')$$
	one of the following happens:\\
	(a) $\dim(G')>k+l$; or\\
	(b) $G'$ contains $\mathbb{R}^k\times K'$ as a proper subgroup, where $K'=\mathrm{Iso}(p',G')$ has dimension $l$ and number of components as $m+1$.
	
	For (a), $G'$ contains $\mathbb{R}^{k+1}$ as a closed subgroup by Lemma \ref{dim_torus}. Thus it contains $\mathbb{R}^k\times\mathbb{Z}$ as a closed subgroup.
	
	For (b), by Proposition \ref{dim_cpt}, any element of $G'$ outside $\mathbb{R}^k\times K'$ has infinite order. Hence $G'$ contains a closed subgroup $\mathbb{R}^k\times\mathbb{Z}$ as well.
	
	As we did before, we can further assume that the extra $\mathbb{Z}$ subgroup in $G'$ has generator whose displacement at $p'$ is less than $1$. Let $\delta(n,\eta)>0$ be the constant in Lemma \ref{dim_gap}. We consider
	\begin{align*}
		S_i:=\{\  1\le s\le r_i\ |\ & d_{GH}((s{M}_i,{p}_i,\Gamma_i),(Y,q,H))\le \delta/3 \text{ for some space } (Y,q,H)\\
		&  \text{ satisfying the following conditions}\\
		& \textit{(C1)}\ H \text{ contains $\mathbb{R}^{k}\times\mathbb{Z}$ as a closed subgroup},\\
		& \textit{(C2)}\ \text{this extra $\mathbb{Z}$ subgroup of $H$ has generator whose }\\
		&\ \ \ \ \text{\ \ \ displacement at $q$ is less than $1$.}\}
	\end{align*}
	$S_i$ is not empty because $r_i\in S_i$ for $i$ large. We pick $s_i\in S_i$ with $$\inf(S_i)\le s_i\le\inf(S_i)+1/i.$$
	
    By Lemma \ref{dim_gap} and the same argument we applied before,  we conclude that $s_i\to\infty$. By our choice of $s_i$, we also have $r_i/s_i\to\infty$.
	
	We consider
	$$(s_i{M}_i,{p}_i,\Gamma_i)\overset{GH}\longrightarrow(Y_\infty,q_\infty,H_\infty).$$
	If $\dim(H_\infty)>k+l$, then it contains $\mathbb{R}^k\times\mathbb{Z}$ as a closed subgroup, and we get a contradiction by scaling down $s_i$ by a constant. If $\dim(H_\infty)<k+l$, or $\dim(H_\infty)=k+l$ but $\mathrm{Iso}(q_\infty,H_\infty)$ has dimension $<l$, or $\dim(H_\infty)=k+l$ with $\dim(\mathrm{Iso}(q_\infty,H_\infty))=l$ but number of connected components of $\mathrm{Iso}(q_\infty,H_\infty)$ being less than $m+1$, then we consider
	$$(s_i{M}_i,{p}_i,\Gamma_i)\overset{GH}\longrightarrow(Y_\infty,q_\infty,H_\infty)$$
	and its rescaling sequence ($r_i/s_i\to\infty$)
	$$(r_i{M}_i,{p}_i,\Gamma_i)\overset{GH}\longrightarrow({X}',{p}',G').$$
	Apply the induction assumptions and passing to a tangent cone at $p'$, we rule out these cases.
	
	The only remaining case is $(H_\infty)_0=\mathbb{R}^{k}\times \mathbb{T}^l$ with $\mathbb{T}^l$ fixing $q_\infty$ and $\mathrm{Iso}(q_\infty,H_\infty)$ having at least $m+1$ many components. According to Proposition \ref{dim_cpt}, $\mathrm{Iso}(q_\infty,H_\infty)$ has exactly $m+1$ many components. If $\#\pi_0(H_\infty)$ is finite, then by Proposition \ref{dim_cpt} again, $\#\pi_0(H_\infty)=m+1$ and $H_\infty=\mathbb{R}^k\times\mathrm{Iso}(q_\infty,H_\infty)$. Apply Lemmas \ref{dim_ind_untwisted} and \ref{dim_gap} here, we result in a desired contradiction. If $\#\pi_0(H_\infty)=\infty$, then $H_\infty$ contains a closed subgroup $\mathbb{R}^k\times\mathbb{Z}$, and we can scale down $s_i$ by a suitable constant to rule out this case.
\end{proof}

We proof some corollaries to end this section, which will be used in Section 4 to bound the number of short generators.

In the triple induction proof, recall that for a space $(X,p,G)$ with $G=\mathbb{R}^k\times \mathrm{Iso}(p,G)$, we have defined $\dim_R(G)=k$ and $\dim_T(G)=\dim(\mathrm{Iso}(p,G))$. One can regard the tuple $(\dim_T(G),\dim_R(G),\#\pi_0(G))$ as an order on the set of these spaces. We introduce a similar notion for general group actions.

\begin{defn}\label{defn_order}
	Let $(X,p,G)$ be a space. We denote $\overline{G}$ as the subgroup generated by $G_0$ and $\mathrm{Iso}(p,G)$. We define $\dim_T(\overline{G})=\dim(\mathrm{Iso}(p,G))$ and $\dim_R(\overline{G})=\dim(G)-\dim_T(\overline{G})$.
\end{defn}

\begin{defn}\label{def_order}
Let $(Y_1,q_1,H_1)$ and $(Y_2,q_2,H_2)$ be two spaces. We say that $$(Y_1,q_1,H_1)\lesssim(Y_2,q_2,H_2),$$ if one of the following holds:\\
(1) $\dim_T (\overline{H}_1)\le\dim_T(\overline{H}_2)$;\\
(2) $\dim_T (\overline{H}_1)=\dim_T(\overline{H}_2)$, $\dim_R (\overline{H}_1)\le\dim_R(\overline{H}_2)$;\\
(3) $\dim_T (\overline{H}_1)=\dim_T(\overline{H}_2)$, $\dim_R (\overline{H}_1)=\dim_R(\overline{H}_2)$, $\#\pi_0(\overline{H}_1)\le\#\pi_0(\overline{H}_2)$.

We say that $$(Y_1,q_1,H_1)\sim (Y_2,q_2,H_2),$$ if $\dim_T (\overline{H}_1)=\dim_T(\overline{H}_2)$, $\dim_R (\overline{H}_1)=\dim_R(\overline{H}_2)$ and $\#\pi_0(\overline{H}_1)=\#\pi_0(\overline{H}_2)$.
\end{defn}

Similarly, we can define $(Y_1,q_1,H_1)<(Y_2,q_2,H_2)$. With respect to this order, the three inductions in the proof of Theorem \ref{dimension}(1) mean that, if Theorem \ref{dimension}(1) holds for all $(X_1,x_1,G_1)$ with $(X_1,x_1,G_1)<(X,x,G)$, then it holds for $(X,x,G)$. With this definition, we derive the following Corollary from Theorem \ref{dimension} and Proposition \ref{dim_Z}:

\begin{cor}\label{dimension_order}
	Let $(M_i,p_i,\Gamma_i)$ be a sequence with the assumptions in Theorem \ref{dimension}. If the following two sequences converge $(r_i\to\infty)$:
	$$({M}_i,{p}_i,\Gamma_i)\overset{GH}\longrightarrow ({X},{p},G),$$
	$$(r_i{M}_i,{p}_i,\Gamma_i)\overset{GH}\longrightarrow ({X}',{p}',G'),$$
	then $(X',p',G')\lesssim (X,p,{G})$. Moreover, if $\sim$ holds, then $G'=\overline{G'}$.
\end{cor}

\begin{proof}
	The first part follows from Theorem \ref{dimension}.
	
	For the second part, when $(X',p',G')\sim(X,p,G)$, by definition, this means that $\overline{G}$ and $\overline{G'}$ share the same $\dim_T$, $\dim_R$, and $\# \pi_0$. The result $G'=\overline{G'}$ follows from Proposition \ref{dim_Z}.
\end{proof}

\begin{rem}
	Notice that Theorem \ref{dimension} can eliminate $G=S^1$ fixing base point with $G'=\mathbb{R}^2$, while Corollary \ref{dimension_order} cannot. However, Corollary \ref{dimension_order} is sufficient for the argument in next section and streamlines the proof (see proof of Theorem \ref{main_sg} in Section 4).
\end{rem}

\begin{cor}\label{dimension_cor}
	Let $(M_i,p_i,\Gamma_i)$ be a sequence with the assumptions in Theorem \ref{dimension} and $H_i$ be a subgroup of $\Gamma_i$ for each $i$. Suppose that the following two sequences converge $(r_i\to\infty)$:
	$$({M}_i,{p}_i,\Gamma_i,H_i)\overset{GH}\longrightarrow ({X},{p},G,H),$$
	$$(r_i{M}_i,{p}_i,\Gamma_i,H_i)\overset{GH}\longrightarrow ({X}',{p}',G',H')$$
	If $\overline{G}=\overline{H}$ and $H'$ is a proper subgroup of $G'$, then $(X',p',H')<(X,p,G)$.
\end{cor}

\begin{proof}
	By the assumption $\overline{G}=\overline{H}$ and Corollary \ref{dimension_order},
	$$(X',p',H')\lesssim (X',p',G')\lesssim (X,p,{G});$$
	$$(X',p',H')\lesssim (X,p,{H})\sim (X,p,{G}).$$
	Suppose that $(X',p',H')\sim(X,p,G)$ happens, then
	$$(X',p',H')\sim(X,p,H),\quad (X',p',G')\sim (X,p,G).$$
	By the second part of Corollary \ref{dimension_order}, we conclude $H'=\overline{H'}$ and $G'=\overline{G'}$. Since $H'$ is a proper subgroup of $G'$, $\overline{H'}$ is proper in $\overline{G'}$. It follows that
	$$(X',p',H')<(X',p',G').$$
	On the other hand,
	$$(X',p',H')\sim (X',p',G')\sim(X,p,{G}),$$
	a contradiction.
\end{proof}

\begin{rem}\label{rem_order}
	Later in Section 4, we bound the number of short generators by induction on the order introduced in Definition \ref{def_order}. Notice that for any space $(X,x,G)$, if there is a series of spaces
	$$(X,x,G)>(X_1,x_1,G_1)>(X_2,x_2,G_2)>...>(X_i,x_i,G_i)>...,$$
	then this series must stop at certain $k$, that is, $\overline{G_k}=\{e\}$.
\end{rem}

\section{Finite generation}
We prove Theorems \ref{main_sg} by applying Theorem \ref{dimension}. We mention that one can use Theorem \ref{dimension'} instead of \ref{dimension} to bound the number of short generators with a no small almost subgroup assumption around the base point.

\begin{thm}\label{sg_nsas}
	Given $n,R,\epsilon,\eta>0$, there exists a constant $C(n,R,\epsilon,\eta)$ such that the following holds.
	
	Let $(M,p)$ be a complete $n$-manifold with abelian fundamental group and
	$$\mathrm{Ric}\ge -(n-1).$$
	If $\pi_1(M,p)$-action on the Riemannian universal cover $(\widetilde{M},\tilde{p})$ has no $\epsilon$-small $\eta$-subgroup at $q$ with scale $r\in(0,1]$ for all $q\in B_1(\tilde{p})$, then $\#S(p,R)\le C(n,R,\epsilon,\eta)$.
\end{thm}

\begin{thm}\label{milnor_nsas}
	Let $(M,p)$ be an open $n$-manifold with $\mathrm{Ric}\ge 0$. If there are $\epsilon,\eta>0$ such that the $\pi_1(M,p)$-action on the Riemannian universal cover $\widetilde{M}$ has no $\epsilon$-small $\eta$-subgroup at $q$ with scale $r>0$ for all $q\in\widetilde{M}$, then $\pi_1(M)$ is finitely generated.
\end{thm}

As indicated before, we only focus on Theorems \ref{main_sg} and \ref{main_milnor} in this paper.

\begin{proof}[Proof of Theorems \ref{main_sg}]
	Suppose that there exists a contradicting convergent sequence of $n$-manifolds with $\mathrm{Ric}_{M_i}\ge -(n-1)$ and $\mathrm{vol}(B_1(\tilde{p}_i))\ge v>0$
	\begin{center}
		$\begin{CD}
		(\widetilde{M}_i,\tilde{p}_i,\Gamma_i) @>GH>> (\widetilde{X},\tilde{p},G)\\
		@VV\pi_iV @VV\pi V\\
		(M_i,p_i) @>GH>> (X,p)
		\end{CD}$
	\end{center}
	satisfying the following conditions:\\
	(1) $\Gamma_i$ can be generated by loops of length less than $R$,\\
	(2) $\#S(p_i)\ge 2^i$,\\
	(3) there is a positive function $\Phi$ such that $\Gamma_i$-action is scaling $\Phi$-nonvanishing at $\tilde{p}_i$ for all $i$.
	
	To derive a contradiction, the goal is to show that $\#S(p_i)\le N$ for all $i$ large. We rule out such contradicting sequence above by induction on the order of limit space $(\widetilde{X},\tilde{p},G)$ (see Definition \ref{defn_order} and Remark \ref{rem_order}).
	
	If $G$ is discrete, then by Corollary \ref{stable_nss}, there is $N$ such that $\#\Gamma_i(R)\le N$ for all $i$ large. In particular, $\#S(p_i)$ cannot diverge to infinity, a contradiction.
	
	Assuming that the statement holds for all possible limit spaces $(\widetilde{X}_1,\tilde{x}_1,G_1)$ with $$(\widetilde{X}_1,\tilde{x}_1,G_1)<(\widetilde{X},\tilde{p},G),$$
	we show that it also holds for $(\widetilde{X},\tilde{p},G)$.
	
	Given each $\epsilon>0$, by basic properties of short basis and Bishop-Gromov relative volume comparison, the number of short generators with length between $\epsilon$ and $R$ is bounded by some constant $C(n,R,\epsilon)$. Thus the number of short generators with length less than $\epsilon$ is larger than $2^i-C(n,R,\epsilon)\to\infty$. By a diagonal argument and passing to a subsequence, we can pick $\epsilon_i\to 0$ such that number of short generators with length less than $\epsilon_i$ is larger than $2^i$. Replacing $M_i$ by $\widetilde{M}_i/\langle\Gamma_i(\epsilon_i)\rangle$,we can assume that $\Gamma_i=\langle\Gamma_i(\epsilon_i)\rangle$.
	
	We introduce some notations here. For an integer $m$, we denote $\gamma_{i,m}$ as the $m$-th short generator of $\Gamma_i$. For a sequence $m_i\to\infty$ below, we always assume that $m_i\le \#S(p_i)$. We consider $H_i$ as the subgroup in $\Gamma_i$ generated by first $m_i$ short generators and $H$ as a limit group of $H_i$:
	$$(\widetilde{M}_i,\tilde{p}_i,H_i)\overset{GH}\longrightarrow(\widetilde{X},\tilde{p},H).$$
	
	\textit{Case 1: There is a sequence $m_i\to\infty$ such that $(\widetilde{X},\tilde{p},H)<(\widetilde{X},\tilde{p},G)$.}
	
	If this happens, we replace $M_i$ by $\widetilde{M_i}/\Gamma_{i,m_i}$ and finish the induction step.
	
	\textit{Case 2: For any sequence $m_i\to\infty$, $(\widetilde{X},\tilde{p},H)\sim(\widetilde{X},\tilde{p},G)$.}
	
	We pass to tangent cone of $\widetilde{X}$ at $\tilde{p}$ (see Corollary \ref{dim_tangent}). By a standard diagonal argument, there is some $s_i\to\infty$ slowly such that $\epsilon_is_i\to 0$ and
	$$(s_i\widetilde{M}_i,\tilde{p}_i,\Gamma_i,H_i)\overset{GH}\longrightarrow(C_{\tilde{p}}\widetilde{X},\tilde{o},G_{\tilde{p}},H_{\tilde{p}}).$$
	Without lose of generality, we can assume that $G_{\tilde{p}}=H_{\tilde{p}}$ here. Otherwise, 
	$$(C_{\tilde{p}}\widetilde{X},\tilde{o},H_{\tilde{p}})<(\widetilde{X},\tilde{p},G)$$
	and we can apply the induction assumption to rule out such a sequence. We replace $M_i$ by $s_iM_i$ and continue the proof.
	
	Now we have
	$$(\widetilde{M}_i,\tilde{p}_i,\Gamma_i,H_i)\overset{GH}\longrightarrow (\widetilde{X},\tilde{p},G,H)$$
	with $G=H$. We consider intermediate coverings $\overline{M}_i=\widetilde{M}_i/H_i$ and $K_i=\Gamma_{i}/H_i$
	$$(\overline{M}_i,\bar{p}_i,K_i)\overset{GH}\longrightarrow(\overline{X},\bar{p},\{e\}).$$
	Together with the fact that $K_i$ is generated by elements with length less than $\epsilon_i\to 0$, we have $$\mathrm{diam}(K_i\cdot\bar{p}_i)\to 0.$$
	
	Put $r_i=\mathrm{diam}(K_i\cdot\bar{p}_i)^{-1}\to\infty$. Rescaling the above sequences by $r_i$ and passing to a subsequence, we obtain the following convergent sequences:

	\begin{center}
		$\begin{CD}
		(r_i\widetilde{M_i},\tilde{p}_i,\Gamma_i,H_i) @>GH>> (\widetilde{X}',\tilde{p}',G',H')\\
		@VVV @VVV\\
		(r_i\overline{M}_i,\bar{p}_i,K_i) @>GH>> (\overline{X}',\bar{p}',\Lambda)
		\end{CD}$
	\end{center}
	with $\mathrm{diam}(\Lambda\cdot \bar{p})=1$. In particular, we conclude that $H'$ is a proper subgroup of $G'$. By Corollary \ref{dimension_cor},
	$$(\widetilde{X}',\tilde{p}',H')<(\widetilde{X},\tilde{p},G).$$
	
	\textbf{Claim :} On $\overline{M}$, $\pi_1(\overline{M}_i,\bar{p}_i)$ can be generated by loops of length less than $1$.
	
	Indeed, $r_i|\gamma_{i,m_i}|\le 1$ because
	\begin{align*}
		r_i^{-1}&=\mathrm{diam}(K_i\cdot\bar{p}_i)\\
		&=\sup_{\gamma\in \Gamma_i} d(\gamma H_i\cdot \tilde{p}_i,H_i\cdot \tilde{p}_i)\\
		&\ge d(\gamma_{i,m_i+1} H_i\cdot \tilde{p}_i,H_i\cdot \tilde{p}_i)\\
		&=d(\gamma_{i,m_i+1} t\cdot \tilde{p}_i,\tilde{p}_i)\ (\text{for some } t\in H_i)\\
		&\ge d(\gamma_{i.m_i}\cdot \tilde{p}_i,\tilde{p}_i).
	\end{align*}
	The last inequality follows from the method by which we select short generators.
	
	Now we have the following new contradicting sequence:
	\begin{center}
		$\begin{CD}
		(r_i\widetilde{M}_i,\tilde{p}_i,\Gamma_{i,m_i}) @>GH>> (\widetilde{X}',\tilde{p}',H')\\
		@VVV @VVV\\
		(r_i\overline{M}_i,\bar{p}_i) @>GH>> (\overline{X}',\bar{p}')
		\end{CD}$
	\end{center}
	with $(\widetilde{X}',\tilde{p}',H')<(X,p,G)$. Applying the induction assumption, we can rule out the existence of such a sequence and complete the proof.
\end{proof}

\begin{rem}
    In the proof above, if $\dim_T(\overline{H'})=\dim_T(\overline{G})$ and $\dim_R(\overline{H'})=\dim_R(\overline{G})$, then
    $$(\widetilde{X}',\tilde{p}',H')<(\widetilde{X},\tilde{p},G)$$
    means $\#\pi_0(\mathrm{Iso}(\tilde{p}',H'))<\#\pi_0(\mathrm{Iso}(\tilde{p},G))$. Therefore, when the dimension does not decrease, we actually did an induction on the number of connected components of the isotropy subgroup, as mentioned in the introduction.
\end{rem}

Recall that to prove results on the Milnor conjecture, by \cite{Wi00} it suffices to check abelian fundamental groups.

\begin{thm}\cite{Wi00}\label{Wil_red}
	Let $M$ be an open manifold of $\mathrm{Ric}\ge 0$.
	If $\pi_1(M)$ is not finitely generated, then it contains an abelian subgroup, which is not finitely generated.
\end{thm}

Theorem \ref{main_milnor} follows from Theorem \ref{main_sg} by a scaling trick.

\begin{proof}[Proof of Theorem \ref{main_milnor}]
	By Theorem \ref{Wil_red}, we can assume that $\pi_1(M,p)$ is abelian. By assumptions, there is $v>0$ such that
	$$\mathrm{vol}(R^{-1}B_R(\tilde{p}))\ge v>0$$
	for all $R>0$. Let $\{\gamma_1,...,\gamma_i,...\}$ be a set of short generators at $p$. We show that there are at most $C$ many short generators, where $C=C(n,1,v,\Phi)$ is the constant in Theorem \ref{main_sg}. Suppose that there are at least $C+1$ many short generators. We put $R$ as the length of $\gamma_{C+1}$. Then on $(R^{-1}\widetilde{M},\tilde{p})$, $\pi_1(M,p)$-action is scaling $\Phi$-nonvanishing, but there are $C+1$ many short generators of length $\le 1$, which is a contradiction to Theorem \ref{main_sg}.
\end{proof}

\Addresses

\end{document}